\documentclass[11pt]{article}
\usepackage{amsmath,amssymb}
\usepackage{graphicx}
\usepackage{amsfonts}

\usepackage{amsthm}
\newtheorem{thm}{Theorem}[section]

\newtheorem{corollary}[thm]{Corollary}
\newtheorem{define}[thm]{Definition}
\newtheorem{assume}[thm]{Assumption}

\usepackage[USenglish]{babel} 
\usepackage[T1]{fontenc}
\usepackage[ansinew]{inputenc}
\usepackage{mathrsfs}
\usepackage{stmaryrd} 
\usepackage{verbatim}
\usepackage{MnSymbol}  
\usepackage{accents}   
\usepackage{bbm}			 

\numberwithin{equation}{section}
\numberwithin{figure}{section}

\title{Variational Partitioned Runge-Kutta methods\\ for Lagrangians linear in velocities}

\author{
Tomasz M. Tyranowski\footnote{\texttt{maximus@caltech.edu}} \quad \qquad Mathieu Desbrun\footnote{\texttt{mathieu@caltech.edu}} \\\\ Computing {\small $+$} Mathematical Sciences \\ California Institute of Technology\\Pasadena, CA 91125, USA
}

\date{}

\begin{document}

\maketitle

\begin{abstract}
In this paper we construct higher-order variational integrators for a class of degenerate systems described by Lagrangians that are linear in velocities. We analyze the geometry underlying such systems and develop the appropriate theory for variational integration. Our main observation is that the evolution takes place on the primary constraint and the \textquoteleft Hamiltonian' equations of motion can be formulated as an index-1 differential-algebraic system. We also construct variational Runge-Kutta methods and analyze their properties. The general properties of Runge-Kutta methods depend on the \textquoteleft velocity' part of the Lagrangian. If the \textquoteleft velocity' part is also linear in the position coordinate, then we show that non-partitioned variational Runge-Kutta methods are equivalent to integration of the corresponding first-order Euler-Lagrange equations, which have the form of a Poisson system with a constant structure matrix, and the classical properties of the Runge-Kutta method are retained. If the \textquoteleft velocity' part is nonlinear in the position coordinate, we observe a reduction of the order of convergence, which is typical of numerical integration of DAEs. We verify our results through numerical experiments for various dynamical systems.

\end{abstract}

\section{Introduction}
\label{sec:intro}

Geometric integrators are numerical methods that preserve geometric structures and properties of the flow of a differential equation. Structure-preserving integrators have attracted considerable interest due to their excellent numerical behavior, especially for long-time integration of equations possessing geometric properties (see \cite{HLWGeometric}, \cite{McLachlanQuispel}, \cite{SanzSerna}).

An important class of structure-preserving integrators are \emph{variational integrators} (see \cite{MarsdenWestVarInt}). This type of numerical schemes is based on discrete variational principles and provides a natural framework for the discretization of Lagrangian systems, including forced, dissipative, or constrained ones. Variational integrators were introduced in the context of finite-dimensional mechanical systems, but were later generalized to Lagrangian field theories (see \cite{MarsdenPatrickShkoller}) and applied in many computations, for example in elasticity (\cite{LewAVI}), electrodynamics (\cite{SternDesbrun}), or fluid dynamics (\cite{Pavlov}).

Theoretical aspects of variational integration are well understood in the case when the Lagrangian describing the considered system is regular, that is, when the corresponding Legendre transform is (at least locally) invertible. However, the corresponding theory for degenerate Lagrangian systems is less developed. The analysis of degenerate systems becomes a little more cumbersome, because the Euler-Lagrange equations may cease to be second order, or may not even make any sense at all. In the latter case one needs to determine if there exists a submanifold of the configuration bundle $TQ$ on which consistent equations of motion can be derived. This can be accomplished by applying the Dirac theory of constraints or the pre-symplectic constraint algorithm (see \cite{GotayPhDThesis}, \cite{MarsdenRatiuSymmetry}). 

A particularly simple case of degeneracy occurs when the Lagrangian is linear in velocities. In that case, the dynamics of the system is defined on the configuration manifold $Q$ itself, rather than its tangent bundle $TQ$, provided that some regularity conditions are satisfied. Such systems arise in many physical applications, including interacting point vortices in the plane (see \cite{Newton}, \cite{RowleyMarsden}, \cite{VankerschaverLeok}), or partial differential equations such as the nonlinear Schr\"{o}dinger (\cite{FaouGeometricSchrodinger}), KdV (\cite{DrazinSolitonsIntroduction}, \cite{GotayKdV}) or Camassa-Holm equations (\cite{CamassaHolm}, \cite{CamassaHolmHyman}). In Section~\ref{sec: Numerical experiments} we show how certain Poisson systems can be recast as Lagrangian systems whose Lagrangians are linear in velocities. Therefore, our approach offers a new perspective on geometric integration of Poisson systems, which often arise as semi-discretizations of integrable nonlinear partial differential equations, e.g., the Toda or Volterra lattice equations, and play an important role in the modeling of many physical phenomena (see \cite{ErgencKarasozen}, \cite{Karasozen}, \cite{Suris}).

This paper is organized as follows. In Section~\ref{sec: Geometric setup} we introduce a proper geometric setup and discuss the properties of systems that are linear in velocities. In Section~\ref{sec: Veselov discretization and Discrete Mechanics} we analyze the general properties of variational integrators and point out how the relevant theory differs from the non-degenerate case. In Section~\ref{sec: Variational partitioned Runge-Kutta methods} we introduce variational partitioned Runge-Kutta methods and discuss their relation to numerical integration of differential-algebraic systems. In Section~\ref{sec: Numerical experiments} we present the results of our numerical experiments for Kepler's problem, a system of two interacting vortices, and the Lotka-Volterra model. We summarize our work and discuss possible extensions in Section~\ref{sec: summary}.

\section{Geometric setup}
\label{sec: Geometric setup}

Let $Q$ be the configuration manifold and $TQ$ its tangent bundle. Throughout this work we will assume that the dimension of the configuration manifold $\dim Q = n$ is even. We will further assume $Q$ is a vector space and by a slight abuse of notation we will denote by $q$ both an element of $Q$ and the vector of its coordinates $q=(q^1,\ldots,q^n)$ in a local chart on $Q$. It will be clear from the context which definition is invoked. Consider the Lagrangian $L:TQ \longrightarrow \mathbb{R}$ given by

\begin{equation}
\label{eq: Linear Lagrangian Intrinsic}
L(v_q)=\langle \alpha, v_q \rangle - H(q),
\end{equation}

\noindent
where $\alpha: Q \longrightarrow T^*Q$ is a smooth one-form, $H:Q \longrightarrow \mathbb{R}$ is the Hamiltonian, and $v_q \in T_q Q$. Let $(q^\mu, \dot q^\mu)$ denote canonical coordinates on $TQ$, where $\mu=1,\ldots,n$. In these coordinates we can consider

\begin{equation}
\label{eq: Linear Lagrangian in Coordinates}
L(q,\dot q)=\alpha_\mu (q) \, \dot q^\mu - H(q),
\end{equation} 

\noindent
where summation over repeated Greek indices is implied.

\subsection{Equations of motion}
\label{sec: Equations of motion}

The Lagrangian \eqref{eq: Linear Lagrangian Intrinsic} is degenerate, since the associated Legendre transform 

\begin{equation}
\label{eq: Legendre Transform}
\mathbb{F}L:TQ \ni v_q \longrightarrow \alpha_q \in T^*Q
\end{equation}

\noindent
is not invertible. The local representation of the Legendre transform is

\begin{align}
\label{eq: Legendre Transform in coordinates}
\mathbb{F}L(q^\mu, \dot q^\mu) = \bigg(q^\mu, \frac{\partial L}{\partial \dot q^\mu} \bigg) = \big( q^\mu,\alpha_\mu (q) \big),
\end{align}

\noindent
that is,

\begin{equation}
\label{eq: Legendre Transform in coordinates - just momentum}
p_\mu = \alpha_\mu (q),
\end{equation}

\noindent
where $(q^\mu,p_\mu)$ denote canonical coordinates on $T^*Q$. The dynamics is defined by the action functional

\begin{equation}
\label{eq: action}
S[q(t)] = \int_{a}^{b} L\big(q(t),\dot q(t)\big)\,dt
\end{equation}

\noindent 
and Hamilton's principle, which seeks the curves $q(t)$ such that the functional $S[q(t)]$ is stationary under variations of $q(t)$ with fixed endpoints, i.e., we seek $q(t)$ such that

\begin{align}
\label{eq:Hamilton's principle}
dS[q(t)] \cdot \delta q(t)=\frac{d}{d\epsilon} \bigg|_{\epsilon=0}S[q_\epsilon(t)]=0
\end{align}

\noindent
for all $\delta q(t)$ with $\delta q(a)=\delta q(b)=0$, where $q_\epsilon (t)$ is a smooth family of curves satisfying $q_0=q$ and $\frac{d}{d\epsilon} \big|_{\epsilon=0} q_\epsilon = \delta q$. The resulting Euler-Lagrange equations

\begin{equation}
\label{eq:E-L ODE}
M_{\mu \nu}(q) \, \dot q^\nu = \partial_\mu H(q)
\end{equation}

\noindent
form a system of \emph{first-order} ODEs, where we assume that the even-dimensional antisymmetric matrix $M_{\mu\nu}(q) = \partial_\mu \alpha_\nu (q) - \partial_\nu \alpha_\mu (q)$ is invertible for all $q \in Q$. Without loss of generality we can further assume that the coordinate mapping $p_\mu = \alpha_\mu(q)$ is invertible and its inverse is smooth: if the Jacobian $\partial \alpha_\mu / \partial q^\nu$ is singular, we can redefine $\alpha_\mu(q) \rightarrow \alpha_\mu(q) + b_\mu(q^\mu)$, where $b_\mu(q^\mu)$ are arbitrary functions; the Euler-Lagrange equations remain the same, and with the right choice of the functions $b_\mu(q^\mu)$, the redefined Jacobian can be made nonsingular. Using $B=M^{-1}$, Eq.~\eqref{eq:E-L ODE} can be equivalently written as the Poisson system 

\begin{equation}
\label{eq:E-L Poisson}
\dot q^\mu = B^{\mu \nu}(q) \, \partial_\nu H(q).
\end{equation}

The Euler-Lagrange equations \eqref{eq:E-L ODE} can also be formulated as the implicit \textquoteleft Hamiltonian' system

\begin{align}
\label{eq: Hamiltonian DAE}
p_\mu &= \alpha_\mu(q), \nonumber \\
\dot p_\mu &= \partial_\mu \alpha_\nu (q) \, \dot q^\nu - \partial_\mu H(q).
\end{align}

\noindent
Since the Lagrangian $L$ is degenerate, \eqref{eq: Hamiltonian DAE} is an index-1 system of differential-algebraic equations (DAE), rather than a Hamiltonian ODE system: the Legendre transform is an algebraic equation and has to be differentiated once with respect to time in order to turn this system into \eqref{eq:E-L ODE}. This reflects the fact that the evolution of the considered degenerate system takes place on the \emph{primary constraint} $N = \mathbb{F}L(TQ) \subsetneq T^*Q$. It is easy to see that the primary constraint $N$ is (locally) diffeomorphic to the configuration manifold $Q$, where the diffeomorphism $\eta:Q \ni q \longrightarrow \alpha_q \in N$ is locally, in the coordinates on $T^*Q$, given by

\begin{equation}
\label{eq: Q-N diffeomorphism}
\eta(q) = (q,\alpha(q)),
\end{equation}

\noindent
where by a slight abuse of notation $\alpha(q)=(\alpha_1(q),\ldots,\alpha_n(q))$. This shows that $q^\mu$ can also be used as local coordinates on $N$. Note that $\eta$ is simply the restriction of $\alpha$ to $N$, i.e., $\eta = \alpha |_{Q\longrightarrow N}$.

\subsection{Symplectic forms}
\label{sec: Symplectic forms}

The spaces $Q$, $TQ$, $T^*Q$ and $N$ can be equipped with several symplectic or pre-symplectic forms. It is instructive to investigate the relationships between them in order to later avoid confusion regarding the sense in which variational integrators for Lagrangians linear in velocities are symplectic. On the configuration space $Q$ we can define the two-form

\begin{equation}
\label{eq: Symplectic Form on Q}
\Omega = -d \alpha,
\end{equation}

\noindent
which in local coordinates can be expressed as

\begin{equation}
\label{eq: Symplectic Form on Q in coordinates}
\Omega = -d\alpha_\mu \wedge dq^\mu = -M_{\mu\nu}(q) \, dq^\mu \otimes dq^\nu.
\end{equation}

\noindent
The two-form $\Omega$ is symplectic if it is nondegenerate, i.e., if the matrix $M_{\mu\nu}$ is invertible for all $q$. 

The cotangent bundle $T^*Q$ is equipped with the canonical Cartan one-form $\tilde \Theta: T^*Q \longrightarrow T^*T^*Q$, which is intrinsically defined by the formula

\begin{equation}
\label{eq: Cartan form}
\tilde \Theta(\omega) = (\pi_{T^*Q})^* \omega
\end{equation}

\noindent
for any $\omega \in T^*Q$, where $\pi_{T^*Q}: T^*Q \longrightarrow Q$ is the cotangent bundle projection. In canonical coordinates we have

\begin{equation}
\label{eq: Cartan form in coordinates}
\tilde \Theta = p_\mu dq^\mu.
\end{equation}

\noindent
We further have the canonical symplectic two-form

\begin{equation}
\label{eq: Symplectic Form on T*Q}
\tilde \Omega = -d\tilde\Theta = dq^\mu \wedge dp_\mu.
\end{equation}

\noindent
The symplectic forms $\Omega$ and $\tilde \Omega$ are related by

\begin{equation}
\label{eq: Omega vs Omega tilde}
\Omega = \alpha^* \tilde \Omega.
\end{equation}

\noindent
This follows from the simple calculation

\begin{equation}
\label{eq: Theta vs alpha derivation}
\alpha^*\tilde \Theta \cdot v_q = \tilde \Theta(\alpha_q) \cdot T\alpha(v_q) = \alpha_q \cdot T\pi_{T^*Q} \circ T\alpha(v_q) = \alpha_q \cdot T(\pi_{T^*Q} \circ \alpha)(v_q) = \alpha_q \cdot v_q,
\end{equation}

\noindent
where we used \eqref{eq: Cartan form} and the fact that $\pi_{T^*Q} \circ \alpha = \text{id}_Q$. Hence $\alpha^*\tilde \Theta = \alpha$, and taking the exterior derivative on both sides we obtain \eqref{eq: Omega vs Omega tilde}.

Using the Legendre transform \eqref{eq: Legendre Transform} we can define the Lagrangian two-form $\tilde \Omega_L$ on $TQ$ by $\tilde \Omega_L = \mathbb{F}L^*\tilde \Omega$, which in canonical coordinates $(q^\mu,\dot q^\mu)$ is given by

\begin{equation}
\label{eq: Symplectic Form on TQ}
\tilde \Omega_L = dq^\mu \wedge d\alpha_\mu = -M_{\mu\nu}(q) \, dq^\mu \otimes dq^\nu.
\end{equation}

\noindent
The Lagrangian form $\tilde \Omega_L$ is only pre-symplectic, because it is degenerate. Noting that $\mathbb{F}L= \alpha \circ \pi_{TQ}$, where $\pi_{TQ}: TQ \longrightarrow Q$ is the tangent bundle projection, we can relate $\Omega$ and $\tilde \Omega_L$ through the formula

\begin{equation}
\label{eq: Relation between Omega_L and Omega}
\tilde \Omega _L = (\pi_{TQ})^* \alpha^* \tilde \Omega = (\pi_{TQ})^* \Omega.
\end{equation}

The symplectic structure on $N$ can be introduced in two ways: by pushing forward $\Omega$ from $Q$, or pulling back $\tilde \Omega$ from $T^*Q$. Both ways are equivalent

\begin{equation}
\label{eq: Omega_N}
\tilde \Omega_N = \eta_* \Omega = i^* \tilde \Omega,
\end{equation}

\noindent
where $i:N\longrightarrow T^*Q$ is the inclusion map. This follows from the calculation

\begin{equation}
\label{eq: Omega_N derivation}
\eta_* \Omega = (\eta^{-1})^* \alpha^* \tilde \Omega = (\alpha \circ \eta^{-1})^* \tilde \Omega = i^* \tilde \Omega,
\end{equation}

\noindent
where we used $\alpha = i\circ\eta$. If we use $q^\mu$ as coordinates on $N$, then the local representation of $\tilde \Omega_N$ will be given by \eqref{eq: Symplectic Form on Q in coordinates}.

\subsection{Symplectic flows}
\label{sec: Flows}

Let $\varphi_t:Q\longrightarrow Q$ denote the flow of \eqref{eq:E-L ODE} or \eqref{eq:E-L Poisson}. This flow is symplectic on $Q$, that is

\begin{equation}
\label{eq: phi symplectic on Q}
\varphi_t^* \Omega = \Omega.
\end{equation}

\noindent
This fact can be proven by considering the Hamiltonian or Poisson properties of Equation \eqref{eq:E-L ODE} or Equation \eqref{eq:E-L Poisson} (see \cite{HLWGeometric}, \cite{MarsdenRatiuSymmetry}). It also follows directly from the action principle \eqref{eq:Hamilton's principle} (see \cite{RowleyMarsden}). 

Since the Lagrangian \eqref{eq: Linear Lagrangian Intrinsic} is degenerate, the dynamics of the system is defined on $Q$ rather than $TQ$. However, we can obtain the associated flow on $TQ$ through lifting $\varphi_t$ by its tangent map $T\varphi_t: TQ \longrightarrow TQ$. This flow preserves the Lagrangian two-form

\begin{equation}
\label{eq: Tphi symplectic on TQ}
(T\varphi_t)^* \tilde \Omega_L = \tilde \Omega_L.
\end{equation}

\noindent
This can be seen from the calculation

\begin{equation}
\label{eq: Tphi symplectic on TQ derivation}
(T\varphi_t)^* \tilde \Omega_L =(T\varphi_t)^*(\pi_{TQ})^* \Omega=(\pi_{TQ}\circ T\varphi_t)^* \Omega=(\varphi_t \circ \pi_{TQ})^* \Omega=(\pi_{TQ})^*\varphi_t^* \Omega   = \tilde \Omega_L,
\end{equation}

\noindent
where we used \eqref{eq: Relation between Omega_L and Omega}, \eqref{eq: phi symplectic on Q}, and the property $\pi_{TQ}\circ T\varphi_t=\varphi_t \circ \pi_{TQ}$.

The flow $\varphi_t$ induces the flow $\tilde \varphi_t: N \longrightarrow N$ in a natural way as

\begin{equation}
\label{eq: Flow on N}
\tilde \varphi_t = \eta \circ \varphi_t \circ \eta^{-1}.
\end{equation}

\noindent
This flow is symplectic on $N$, i.e., 

\begin{equation}
\label{eq: tilde phi symplectic on N}
\tilde \varphi_t^* \tilde \Omega_N = \tilde \Omega_N,
\end{equation}

\noindent
which can be established through the simple calculation

\begin{equation}
\label{eq: tilde phi symplectic on N derivation}
\tilde \varphi_t^* \tilde \Omega_N = (\eta \circ \varphi_t \circ \eta^{-1})^* \eta_* \Omega = (\eta^{-1})^*\varphi_t^*\eta^*(\eta^{-1})^*\Omega=\eta_*\varphi_t^*\Omega = \tilde \Omega_N,
\end{equation}

\noindent
where we used \eqref{eq: Omega_N} and \eqref{eq: phi symplectic on Q}. The flow $\tilde \varphi_t$ can be interpreted as the symplectic flow for the \textquoteleft Hamiltonian' DAE \eqref{eq: Hamiltonian DAE}.

%
%

\section{Veselov discretization and Discrete Mechanics}
\label{sec: Veselov discretization and Discrete Mechanics}

\subsection{Discrete Mechanics}
\label{sec: Discrete Mechanics}

For a Veselov-type discretization we consider the discrete state space $Q \times Q$, which serves as a discrete approximation of the tangent bundle (see \cite{MarsdenWestVarInt}). We define a discrete Lagrangian $L_d$ as a smooth map $L_d: Q \times Q \longrightarrow \mathbb{R}$ and the corresponding discrete action 

\begin{equation}
\label{eq: Discrete action}
S=\sum_{k=0}^{N-1} L_d(q_k,q_{k+1}).
\end{equation}

\noindent
The variational principle now seeks a sequence $q_0$, $q_1$, $...$, $q_N$ that extremizes $S$ for variations holding the endpoints $q_0$ and $q_N$ fixed. The Discrete Euler-Lagrange equations follow

\begin{equation}
\label{eq:Discrete Euler-Lagrange equations}
D_2L_d(q_{k-1},q_k) + D_1 L_d(q_k,q_{k+1}) = 0.
\end{equation}

\noindent
Assuming that these equations can be solved for $q_{k+1}$, i.e., $L_d$ is non-degenerate, they implicitly define the discrete Lagrangian map $F_{L_d}: Q \times Q \longrightarrow Q \times Q$ such that $F_{L_d}(q_{k-1},q_k)=(q_k,q_{k+1})$. Let $(q^\mu,\bar q^\mu)$ denote local coordinates on $Q\times Q$. We can define the discrete Legendre transforms $\mathbb{F}L_d^+,\mathbb{F}L_d^-:Q\times Q \longrightarrow T^*Q$, which in local coordinates on $Q\times Q$ and $T^*Q$ are respectively given by

\begin{align}
\label{eq: Discrete Legendre Transforms}
\mathbb{F}^+L_d(q,\bar q) &= \big(\bar q, D_2L_d(q,\bar q)\big), \nonumber \\
\mathbb{F}^-L_d(q,\bar q) &= \big(q, -D_1L_d(q,\bar q)\big),
\end{align}

\noindent
where $q=(q^1,\ldots,q^n)$ and $\bar q=(\bar q^1,\ldots,\bar q^n)$. The Discrete Euler-Lagrange equations \eqref{eq:Discrete Euler-Lagrange equations} can be equivalently written as

\begin{equation}
\label{eq:Discrete Euler-Lagrange equations using FL}
\mathbb{F}^+L_d(q_{k-1},q_k)  = \mathbb{F}^-L_d(q_k,q_{k+1}).
\end{equation}

\noindent
Using either of the transforms, one can define the discrete Lagrange two-form on $Q \times Q$ by $\omega_{L_d} = (\mathbb{F}^\pm L_d)^* \tilde \Omega$, which in coordinates gives

\begin{equation}
\label{eq: Discrete Symplectic Form}
\omega_{L_d} = \frac{\partial^2 L_d}{\partial q^\mu \partial \bar q^\nu} dq^\mu \wedge d\bar q^\nu. 
\end{equation}

\noindent
It then follows that the discrete flow $F_{L_d}$ is symplectic, i.e., $F_{L_d}^* \omega_{L_d} = \omega_{L_d}$. Using the Legendre transforms we can pass to the cotangent bundle and define the discrete Hamiltonian map $\tilde F_{L_d}:T^*Q \longrightarrow T^*Q$ by $\tilde F_{L_d} = \mathbb{F}^\pm L_d\circ F_{L_d} \circ (\mathbb{F}^\pm L_d)^{-1}$. This map is also symplectic, i.e., $\tilde F_{L_d}^* \tilde \Omega = \tilde \Omega$.

\subsection{Exact discrete Lagrangian}
\label{sec: Exact discrete Lagrangian}

To relate discrete and continuous mechanics it is necessary to introduce a timestep $h \in \mathbb{R}$. If the continuous Lagrangian $L$ is non-degenerate, it is possible to define a particular choice of discrete Lagrangian which gives an exact correspondence between discrete and continuous systems (see \cite{MarsdenWestVarInt}), the so-called \emph{exact discrete Lagrangian}

\begin{equation}
\label{eq: Exact Discrete Lagrangian}
L^E_d(q,\bar q) = \int_0^h L\big( q_{E}(t), \dot q_{E}(t)\big)\,dt,
\end{equation}

\noindent
where $q_E(t)$ is the solution to the continuous Euler-Lagrange equations associated with $L$ such that it satisfies the boundary conditions $q_E(0)=q$ and $q_E(h)=\bar q$. Note, however, that in the case of a regular Lagrangian the associated Euler-Lagrange equations are second order, therefore boundary value problems are solvable, at least for sufficiently small $h$ and $\bar q$ sufficiently close to $q$. In the case of the Lagrangian \eqref{eq: Linear Lagrangian Intrinsic} the associated Euler-Lagrange equations \eqref{eq:E-L ODE} are first order in time, therefore we have the freedom to choose an initial condition either at $t=0$ or $t=h$, but not both. An exact discrete Lagrangian analogous to \eqref{eq: Exact Discrete Lagrangian} cannot thus be defined on whole $Q \times Q$. We will therefore assume the following definition:

\begin{define}
Let $\Gamma(\varphi_h) = \big\{ \big(q,\varphi_h(q)\big)\in Q\times Q\big\}$ be the graph of $\varphi_h$. The exact discrete Lagrangian $L^E_d:\Gamma(\varphi_h)\longrightarrow \mathbb{R}$ for the Lagrangian \eqref{eq: Linear Lagrangian Intrinsic} is

\begin{equation}
\label{eq: Exact Discrete Lagrangian in Definition}
L^E_d(q,\bar q) = \int_0^h L\big( q_{E}(t), \dot q_{E}(t)\big)\,dt,
\end{equation}

\noindent
where $q_E(t)$ is the solution to \eqref{eq:E-L ODE} that satisfies the initial condition $q_E(0)=q$.
\end{define}

\noindent
Note that in this definition we automatically have $q_E(h)=\bar q$.

\subsection{Singular perturbation problem}
\label{sec: Singular perturbation problem}

As mentioned, the purpose of introducing an exact discrete Lagrangian is to establish an exact correspondence between the continuous and discrete systems. For a regular Lagrangian $L$ and its exact discrete Lagrangian $L^E_d$, one can show that the exact discrete Hamiltonian map $\tilde F_{L^E_d}$ is equal to $\tilde \varphi_h$, where $\tilde \varphi_t$ is the symplectic flow for the Hamiltonian system associated with $L$. The problem is that the exact discrete Lagrangian \eqref{eq: Exact Discrete Lagrangian in Definition} is not defined on the whole space $Q\times Q$, so the discrete Euler-Lagrange equations \eqref{eq:Discrete Euler-Lagrange equations} do not make sense, and it is not entirely clear how to define the associated discrete Lagrangian map $F_{L^E_d}$. One possible way to deal with this issue is to consider a singular perturbation problem. Assume that $Q$ is a Riemannian manifold equipped with the nondegenerate scalar product $\llangle.,.\rrangle$. Define the $\epsilon$-regularized Lagrangian

\begin{equation}
\label{eq: Epsilon-regularized Lagrangian Intrinsic}
L^\epsilon(v_q)=\frac{\epsilon}{2} \llangle v_q,v_q \rrangle + \langle \alpha, v_q \rangle - H(q),
\end{equation}

\noindent
or in coordinates

\begin{equation}
\label{eq: Epsilon-regularized Lagrangian in coordinates}
L^\epsilon(q,\dot q)=\frac{\epsilon}{2} \, g_{\mu\nu} \,\dot q^\mu \,\dot q^\nu + \alpha_\mu (q) \, \dot q^\mu - H(q),
\end{equation}

\noindent
where $g_{\mu\nu}$ denotes the local coordinates of the metric tensor. Without loss of generality assume that in the chosen coordinates $g_{\mu \mu}=1$ and $g_{\mu \nu}=0$ if $\mu \not = \nu$. For $\epsilon>0$ this Lagrangian is nondegenerate and the Legendre transform $\mathbb{F}L^\epsilon:TQ \longrightarrow T^*Q$ is given by

\begin{align}
\label{eq: Regularized Legendre Transform in coordinates}
\mathbb{F}L(q^\mu,\dot q^\mu)= \big( q^\mu,g_{\mu \nu}\,\dot q^\nu +  \alpha_\mu (q) \big),
\end{align}

\noindent
that is,

\begin{align}
\label{eq: Regularized Legendre Transform in coordinates - just momentum}
p_\mu = \epsilon\, g_{\mu \nu}\,\dot q^\nu +  \alpha_\mu (q).
\end{align}

\noindent
The Euler-Lagrange equations

\begin{equation}
\label{eq:Regularized E-L ODE}
\epsilon \, g_{\mu \nu} \, \ddot q^\nu =  M_{\mu \nu}(q) \, \dot q^\nu - \partial_\mu H(q)
\end{equation}

\noindent
are second order. The corresponding Hamiltonian equations (in implicit form) are

\begin{align}
\label{eq: Regularized Hamiltonian ODE}
p_\mu &= \epsilon\, g_{\mu \nu}\,\dot q^\nu + \alpha_\mu(q), \nonumber \\
\dot p_\mu &= \partial_\mu \alpha_\nu (q) \, \dot q^\nu - \partial_\mu H(q).
\end{align}

\noindent
There is no reason to expect that the solutions of \eqref{eq:Regularized E-L ODE} or \eqref{eq: Regularized Hamiltonian ODE} unconditionally approximate the solutions of \eqref{eq:E-L ODE} or \eqref{eq: Hamiltonian DAE}, respectively. The equations \eqref{eq: Regularized Hamiltonian ODE} form a system of first-order ordinary differential equations, and therefore it is possible to specify arbitrary initial conditions $q(0)=q_{init}$ and $p(0)=p_{init}$, whereas initial conditions for \eqref{eq: Hamiltonian DAE} have to satisfy the algebraic constraint $p_{init} = \alpha(q_{init})$. Under certain restrictive analytic assumptions, for some singular perturbation problems it is possible to show that, in order to satisfy the initial conditions, the solutions initially develop a steep boundary layer, but then rapidly converge to the solution of the corresponding DAE system (see \cite{HWODE2}). On the other hand, for other singular perturbation problems, when the initial conditions do not satisfy the algebraic constraint, it may happen that the solutions do not converge to the solution of the DAE, but instead rapidly oscillate (see \cite{Lubich}, \cite{Rabier6}). We expect the latter behavior for \eqref{eq: Regularized Hamiltonian ODE}, as will be demonstrated by a simple example in Section~\ref{sec: Example}. Since our main goal here is to show how the notion of a discrete Legendre transform can be introduced for the exact discrete Lagrangian \eqref{eq: Exact Discrete Lagrangian in Definition}, we will make two intuitive, although nontrivial, assumptions. We refer the interested reader to \cite{HWODE2} and \cite{Lubich} for techniques that can be used to prove these statements rigorously.

\begin{assume}
\label{ass: Initial Value Problem Assumption}
Let $\big(q(t),p(t)\big)$ and $\big(q^\epsilon(t),p^\epsilon(t)\big)$ be the unique smooth solutions of \eqref{eq: Hamiltonian DAE} and \eqref{eq: Regularized Hamiltonian ODE} on the interval $[0,T]$ satisfying the initial conditions $q(0)=q_{init}$, $q^\epsilon(0)=q_{init}$ and $p^\epsilon(0)=p_{init}$, where $p_{init}=\alpha(q_{init})$. Then $q^\epsilon(t) \longrightarrow q(t)$, $p^\epsilon(t) \longrightarrow p(t)$ and $\dot q^\epsilon(t) \longrightarrow \dot q(t)$, $\dot p^\epsilon(t) \longrightarrow \dot p(t)$ uniformly on $[0,T]$ as $\epsilon \longrightarrow 0^+$.
\end{assume}

\begin{assume}
\label{ass: Boundary Value Problem Assumption}
Let $q(t)$ be the unique smooth solution of \eqref{eq:E-L ODE} on the interval $[0,T]$ satisfying the initial condition $q(0)=q_{init}$ and let $q^\epsilon(t)$ be the unique smooth solution of \eqref{eq:Regularized E-L ODE} on the interval $[0,T]$ satisfying the boundary conditions $q^\epsilon(0)=q_{init}$, $q^\epsilon(T)=q_{final}$, where $q_{final}=q(T)$. Then $q^\epsilon(t) \longrightarrow q(t)$ and $\dot q^\epsilon(t) \longrightarrow \dot q(t)$ uniformly on $[0,T]$ as $\epsilon \longrightarrow 0^+$.
\end{assume}

\noindent
With these assumption one can easily see that

\begin{equation}
\label{eq: Exact Discrete Lagrangian as a limit}
L^E_d(q,\bar q) = \lim_{\epsilon\rightarrow 0^+} L^{\epsilon,E}_d(q,\bar q),
\end{equation}

\noindent
where $L^{\epsilon,E}_d$ is the exact discrete Lagrangian for \eqref{eq: Epsilon-regularized Lagrangian Intrinsic}.

\subsection{Exact discrete Legendre transform}
\label{sec: Exact discrete Legendre Transform}

Since $L^\epsilon$ is regular, $L^{\epsilon,E}_d$ is properly defined on the whole space $Q\times Q$ (or at least in a neighborhood of $\Gamma(\varphi_h)$) and the associated exact discrete Legendre transforms satisfy the properties (see \cite{MarsdenWestVarInt})

\begin{align}
\label{eq: Regularized Exact Discrete Legendre Transforms}
\mathbb{F}^+L^{\epsilon,E}_d(q,\bar q) &= \mathbb{F}L^\epsilon\big(q^\epsilon_E(h),\dot q^\epsilon_E(h) \big)=(\bar q, \epsilon \dot{\bar q}^\epsilon+\alpha(\bar q)), \nonumber \\
\mathbb{F}^-L^{\epsilon,E}_d(q,\bar q) &= \mathbb{F}L^\epsilon\big(q^\epsilon_E(0),\dot q^\epsilon_E(0) \big)=(q, \epsilon \dot{q}^\epsilon+\alpha(q)),
\end{align}

\noindent
where $q^\epsilon_E(t)$ is the solution to the regularized Euler-Lagrange equations \eqref{eq:Regularized E-L ODE} satisfying the boundary conditions $q^\epsilon_E(0)=q$ and $q^\epsilon_E(h)=\bar q$, and we denoted $\dot q^\epsilon = \dot q^\epsilon_E(0)$, $\dot{\bar q}^\epsilon = \dot q^\epsilon_E(h)$. In the spirit of \eqref{eq: Exact Discrete Lagrangian as a limit}, we can assume the following definitions of the exact discrete Legendre transforms $\mathbb{F}^{\pm}L^E_d:\Gamma(\varphi_h)\longrightarrow T^*Q$

\begin{align}
\label{eq: Exact Discrete Legendre Transforms as limits}
\mathbb{F}^+L^{E}_d(q,\bar q) &=  \lim_{\epsilon\rightarrow 0^+} \mathbb{F}^+L^{\epsilon,E}_d(q,\bar q) =(\bar q, \alpha(\bar q)), \nonumber \\
\mathbb{F}^-L^{E}_d(q,\bar q) &=  \lim_{\epsilon\rightarrow 0^+} \mathbb{F}^-L^{\epsilon,E}_d(q,\bar q)=(q, \alpha(q)),
\end{align}

\noindent
where $\epsilon \dot{q}^\epsilon \longrightarrow 0$ and $\epsilon \dot{\bar q}^\epsilon \longrightarrow 0$ by uniform convergence of $\dot q^\epsilon_E(t)$. Note that $\mathbb{F}^\pm L^E_d=\alpha \circ \pi^\pm$, where $\pi^+:\Gamma(\varphi_h) \ni (q,\bar q) \longrightarrow \bar q \in Q$ and $\pi^-:\Gamma(\varphi_h) \ni (q,\bar q) \longrightarrow q \in Q$ are projections (both $\pi^\pm$ are diffeomorphisms). This is a close analogy to $\mathbb{F}L=\alpha \circ \pi_{TQ}$ (see Section~\ref{sec: Geometric setup}). We also note the property

\begin{align}
\label{eq: Exact Discrete Legendre Transforms vs Legendre Transform}
\mathbb{F}^+L^E_d(q,\bar q) &= \mathbb{F}L\big(q_E(h),\dot q_E(h) \big), \nonumber \\
\mathbb{F}^-L^E_d(q,\bar q) &= \mathbb{F}L\big(q_E(0),\dot q_E(0) \big),
\end{align}

\noindent
where $q_E(t)$ is the solution of \eqref{eq:E-L ODE} satisfying the initial condition $q_E(0)=q$. This further indicates that our definition of the exact discrete Legendre transforms is sensible. Note that $\mathbb{F}^\pm L^E_d(\Gamma(\varphi_h))=N$. It is convenient to redefine $\mathbb{F}^{\pm}L^E_d:\Gamma(\varphi_h)\longrightarrow N$, that is $\mathbb{F}^\pm L^E_d=\eta \circ \pi^\pm$, so that both transforms are diffeomorphisms between $\Gamma(\varphi_h)$ and $N$.

The discrete Euler-Lagrange equations for $L^E_d$ can be obtained as the limit of the discrete Euler-Lagrange equations for $L^{\epsilon,E}_d$, that is, one can substitute $L^{\epsilon,E}_d$ in \eqref{eq:Discrete Euler-Lagrange equations using FL} and take the limit $\epsilon \longrightarrow 0^+$ on both sides to obtain

\begin{equation}
\label{eq:Discrete Euler-Lagrange equations for the exact discrete Lagrangian}
\mathbb{F}^+L^E_d(q_{k-1},q_k)  = \mathbb{F}^-L^E_d(q_k,q_{k+1}).
\end{equation}

\noindent
This equation implicitly defines the exact discrete Lagrangian map $F_{L^E_d}: \Gamma(\varphi_h) \ni (q_{k-1},q_k) \longrightarrow (q_k,q_{k+1}) \in \Gamma(\varphi_h)$, which, given our definitions, necessarily takes the form $F_{L^E_d}(q_{k-1},q_k)=(q_k,\varphi_h(q_k))$. Using the discrete Legendre transforms $\mathbb{F}^\pm L^E_d$ we can define the corresponding exact discrete \textquoteleft Hamiltonian' map $\tilde F_{L^E_d}: N \longrightarrow N$ as $\tilde F_{L^E_d} = \mathbb{F}^\pm L^E_d\circ F_{L^E_d} \circ (\mathbb{F}^\pm L^E_d)^{-1}$. The simple calculation

\begin{equation}
\label{eq: Exact Discrete Hamiltonian Map}
\tilde F_{L^E_d} = \eta \circ \pi^\pm\circ F_{L^E_d} \circ (\pi^\pm)^{-1}\circ \eta^{-1} = \eta \circ \varphi_h \circ \eta^{-1} = \tilde \varphi_h
\end{equation}

\noindent
shows that the discrete \textquoteleft Hamiltonian' map associated with the exact discrete Lagrangian $L^E_d$ is equal to the \textquoteleft Hamiltonian' flow $\tilde \varphi_h$ for \eqref{eq: Hamiltonian DAE}, i.e., the evolution of the discrete systems described by $L^E_d$ coincides with the evolution of the continuous system described by $L$ at times $t_k = kh$, $k=0,1,2,\ldots$ 

\subsection{Example}
\label{sec: Example}

Let us illustrate these ideas with a very simple example for which analytic solutions are known. Let $Q=\mathbb{R}^2$ and let $(x,y)$ denote local coordinates on $Q$. The tangent bundle is $TQ=\mathbb{R}^2\times \mathbb{R}^2$, and the induced local coordinates are $(x,y,\dot x,\dot y)$. Consider the Lagrangian

\begin{equation}
\label{eq: Example Lagrangian}
L(x,y,\dot x,\dot y)=\frac{1}{2} y \dot x - \frac{1}{2} x \dot y.
\end{equation}

\noindent
The corresponding Euler-Lagrange equations \eqref{eq:E-L ODE} are simply

\begin{align}
\label{eq: Example Euler-Lagrangian Equations}
\dot x &= 0, \nonumber \\
\dot y &= 0, 
\end{align}

\noindent
so the flow $\varphi_t:Q \longrightarrow Q$ is the identity, i.e., $\varphi_t(x,y)=(x,y)$. Let $(x,y,p_x,p_y)$ denote canonical coordinates on the cotangent bundle $T^*Q \cong \mathbb{R}^2\times \mathbb{R}^2$. The Legendre transform is

\begin{equation}
\label{eq: Example Legendre transform}
\mathbb{F}L(x,y,\dot x,\dot y) = \Big(x,y,\frac{1}{2} y,-\frac{1}{2} x\Big).
\end{equation}

\noindent
Let $h$ be a timestep. Note $\Gamma(\varphi_h)=\{(x,y,x,y) \,|\,(x,y) \in Q\}$. The exact discrete Lagrangian \eqref{eq: Exact Discrete Lagrangian in Definition} is therefore 

\begin{equation}
\label{eq: Example Exact Discrete Lagrangian}
L^E_d(x,y,x,y) = 0.
\end{equation}

\noindent
Let us now consider the $\epsilon$-regularized Lagrangian

\begin{equation}
\label{eq: Example Regularized Lagrangian}
L^\epsilon(x,y,\dot x,\dot y)=\frac{\epsilon}{2}\dot x^2+\frac{\epsilon}{2}\dot y^2+\frac{1}{2} y \dot x - \frac{1}{2} x \dot y.
\end{equation}

\noindent
The corresponding Euler-Lagrange equations \eqref{eq:Regularized E-L ODE} take the form

\begin{align}
\label{eq: Example Regularized Euler-Lagrangian Equations}
\epsilon \ddot x + \dot y &= 0, \nonumber \\
\epsilon \ddot y + \dot x &= 0. 
\end{align}

\noindent
One can easily verify analytically that

\begin{align}
\label{eq: Example BVP solution for the singular perturbation problem}
x^\epsilon(t) = \frac{1}{2}\bigg[ (x_i+x_f)-(y_f-y_i) \cot \frac{T}{2 \epsilon}\bigg] &+\frac{1}{2}\bigg[ (y_f-y_i)+(x_f-x_i) \cot \frac{T}{2 \epsilon}\bigg] \sin \frac{t}{\epsilon} \nonumber \\
 &- \frac{1}{2}\bigg[ (x_f-x_i)-(y_f-y_i) \cot \frac{T}{2 \epsilon}\bigg] \cos \frac{t}{\epsilon}, \nonumber \\
y^\epsilon(t) = \frac{1}{2}\bigg[ (y_i+y_f)+(x_f-x_i) \cot \frac{T}{2 \epsilon}\bigg] &-\frac{1}{2}\bigg[ (x_f-x_i)-(y_f-y_i) \cot \frac{T}{2 \epsilon}\bigg] \sin \frac{t}{\epsilon} \nonumber \\
 &- \frac{1}{2}\bigg[ (y_f-y_i)+(x_f-x_i) \cot \frac{T}{2 \epsilon}\bigg] \cos \frac{t}{\epsilon},
\end{align}

\noindent
is the solution to \eqref{eq: Example Regularized Euler-Lagrangian Equations} satisfying the boundary conditions $(x^\epsilon(0),y^\epsilon(0))=(x_i,y_i)$ and $(x^\epsilon(T),y^\epsilon(T))=(x_f,y_f)$. Note that if $x_i\not = x_f$ or $y_i\not = y_f$, then as $\epsilon \longrightarrow 0^+$ this solution is rapidly oscillatory and not convergent. However, if $(x_f,y_f)=\varphi_T(x_i,y_i)=(x_i,y_i)$ (cf. Assumption~\ref{ass: Boundary Value Problem Assumption}) then we have

\begin{align}
\label{eq: Example convergent BVP solution}
x^\epsilon(t) &= x_i, \nonumber \\
y^\epsilon(t) &= y_i, 
\end{align}

\noindent
and this solution converges uniformly (in this simple example it is in fact equal) to the solution of \eqref{eq: Example Euler-Lagrangian Equations} with the same initial condition. We can also find an analytic expression for the exact discrete Lagrangian \eqref{eq: Exact Discrete Lagrangian} associated with \eqref{eq: Example Regularized Lagrangian} as

\begin{equation}
\label{eq: Example Regularized Exact Discrete Lagrangian}
L^{\epsilon,E}_d(x,y,\bar x, \bar y) = \frac{\bar x y - x \bar y}{2}+\frac{(\bar x-x)^2+(\bar y-y)^2}{4} \cot \frac{T}{2 \epsilon}.
\end{equation}

\noindent
Restricting the domain to $\Gamma(\varphi_h)$ we get $L^{\epsilon,E}_d(x,y,x,y)=0$, and comparing to \eqref{eq: Example Exact Discrete Lagrangian} we verify that \eqref{eq: Exact Discrete Lagrangian as a limit} indeed holds. The discrete Legendre transforms \eqref{eq: Discrete Legendre Transforms} associated with $L^{\epsilon,E}_d$ take the form

\begin{align}
\label{eq: Example Regularized Exact Discrete Legendre transforms}
\mathbb{F}^+L^{\epsilon,E}_d(x,y,\bar x, \bar y)&=\bigg(\bar x, \bar y, \frac{y}{2} + \frac{\bar x - x}{2}\cot \frac{T}{2 \epsilon}, -\frac{x}{2} + \frac{\bar y - y}{2}\cot \frac{T}{2 \epsilon}\bigg), \nonumber \\
\mathbb{F}^-L^{\epsilon,E}_d(x,y,\bar x, \bar y)&=\bigg(x, y, \frac{\bar y}{2} + \frac{\bar x - x}{2}\cot \frac{T}{2 \epsilon}, -\frac{\bar x}{2} + \frac{\bar y - y}{2}\cot \frac{T}{2 \epsilon}\bigg).
\end{align}

\noindent
Restricting the domain to $\Gamma(\varphi_h)$ and taking the limit $\epsilon \longrightarrow 0^+$ as in \eqref{eq: Exact Discrete Legendre Transforms as limits}, we can define the exact discrete Legendre transforms associated with \eqref{eq: Example Exact Discrete Lagrangian}
 
\begin{align}
\label{eq: Example Exact Discrete Legendre transforms}
\mathbb{F}^+L^E_d(x,y,x,y)&=\Big(x, y, \frac{y}{2}, -\frac{x}{2} \Big), \nonumber \\
\mathbb{F}^-L^E_d(x,y,x,y)&=\Big(x, y, \frac{y}{2}, -\frac{x}{2} \Big).
\end{align} 

\noindent
Comparing with \eqref{eq: Example Legendre transform}, we see that the property \eqref{eq: Exact Discrete Legendre Transforms vs Legendre Transform} is satisfied, which replicates the analogous property for regular Lagrangians.

\subsection{Variational error analysis}
\label{sec: Variational error analysis}

For a given continuous system described by the Lagrangian $L$, a variational integrator is constructed by choosing a discrete Lagrangian $L_d$ which approximates the exact discrete Lagrangian $L^E_d$. We can define the order of accuracy of the discrete Lagrangian in a way similar to that for discrete Lagrangians resulting from regular continuous Lagrangians (see \cite{MarsdenWestVarInt}).

\begin{define}
A discrete Lagrangian $L_d:Q \times Q \longrightarrow \mathbb{R}$ is of order $r$ if there exists an open subset $U \subset Q$ with compact closure and constants $C>0$ and $\bar h>0$ such that

\begin{equation}
\label{eq: Order of a discrete Lagrangian}
\Big|L_d\big(q(0),q(h) \big) - L^E_d\big(q(0),q(h) \big) \Big| \leq C h^{r+1}
\end{equation}

\noindent
for all solutions $q(t)$ of the Euler-Lagrange equations \eqref{eq:E-L ODE} with initial conditions $q(0) \in U$ and for all $h \leq \bar h$.
\end{define}

We will always assume that the discrete Lagrangian $L_d$ is non-degenerate, so that the discrete Euler-Lagrange equations \eqref{eq:Discrete Euler-Lagrange equations} can be solved for $q_{k+1}$. This defines the discrete Lagrangian map $F_{L_d}:Q\times Q \longrightarrow Q\times Q$ and the associated discrete Hamiltonian map $\tilde F_{L_d}: T^*Q \longrightarrow T^*Q$, as in Section~\ref{sec: Discrete Mechanics}. Of particular interest is the rate of convergence of $\tilde F_{L_d}$ to $\tilde \varphi_h$. One usually considers a \emph{local error} (error made after one step) and a \emph{global error} (error made after many steps). We will assume the following definitions, which are appropriate for differential-algebraic systems (see \cite{HLWGeometric}, \cite{HWODE1}, \cite{HWODE2}, \cite{MarsdenWestVarInt}).

\begin{define}
A discrete Hamiltonian map $\tilde F_{L_d}$ is of order $r$ if there exists an open set $U\subset N$ and constants $C>0$ and $\bar h >0$ such that

\begin{equation}
\label{eq: local error}
\big\|\tilde F_{L_d}(q,p)- \tilde \varphi_h(q,p) \big\| \leq C h^{r+1} 
\end{equation}

\noindent
for all $(q,p)\in U$ and $h\leq \bar h$.
\end{define}

\begin{define}
\label{thm: Definition of the order of convergence}
A discrete Hamiltonian map $\tilde F_{L_d}$ is convergent of order $r$ if there exists an open set $U\subset N$ and constants $C>0$, $\bar h >0$ and $\bar T >0$ such that

\begin{equation}
\label{eq: global error}
\big\|(\tilde F_{L_d})^K(q,p)- \tilde \varphi_T(q,p) \big\| \leq C h^{r+1}, 
\end{equation}

\noindent
where $h = T/K$, for all $(q,p)\in U$, $h\leq \bar h$, and $T \leq \bar T$.
\end{define}

\noindent
If the Lagrangian $L$ is regular, then one can show that a discrete Lagrangian $L_d$ is of order $r$ if and only if the corresponding Hamiltonian map $\tilde F_{L_d}$ is of order $r$ (see \cite{MarsdenWestVarInt}). Also, the associated Hamiltonian equations are a set of ordinary differential equations, and under some smoothness assumptions one can show that if $\tilde F_{L_d}$ is of order $r$, then it is also convergent of order $r$ (see \cite{HWODE1}). However, in the case of the Lagrangian \eqref{eq: Linear Lagrangian Intrinsic} it is not true in general---both the order of the discrete Lagrangian and the local order of the discrete Hamiltonian map may be different than the actual global order of convergence (see \cite{HWODE2}, \cite{HLLectureNotes}), as will be demonstrated in Section~\ref{sec: Variational partitioned Runge-Kutta methods}.

\paragraph{Example: Midpoint Rule.} In a simple example we will demonstrate that the variational order of accuracy of a discretization method is unaffected by a degeneracy of a Lagrangian $L$. In order to calculate the order of a discrete Lagrangian $L_d$, we can expand $L_d(q(0),q(h))$ in a Taylor series in $h$ and compare it to the analogous expansion for $L^E_d$. If the two expansions agree up to $r$ terms, then $L_d$ is of order $r$. Expanding $q(t)$ in a Taylor series about $t=0$ and substituting it in \eqref{eq: Exact Discrete Lagrangian in Definition}, we get the expression

\begin{equation}
\label{eq: Series expansion for exact discrete Lagrangian}
L^E_d\big(q(0),q(h) \big) = h L + \frac{h^2}{2} \bigg(\frac{\partial L}{\partial q} \dot q + \frac{\partial L}{\partial \dot q} \ddot q \bigg) + \frac{h^3}{6} \bigg(\frac{\partial L}{\partial q} \ddot q + \frac{\partial L}{\partial \dot q} \dddot q + \dot q^T \frac{\partial^2 L}{\partial q^2} \dot q + 2 \dot q^T \frac{\partial^2 L}{\partial q \partial \dot q} \ddot q  + \ddot q^T \frac{\partial^2 L}{\partial \dot q^2} \ddot q\bigg) + o(h^3),
\end{equation}

\noindent
where we denoted $q=q(0)$, $\dot q = \dot q(0)$, etc., and the Lagrangian L and its derivatives are computed at $(q,\dot q)$. For the Lagrangian \eqref{eq: Linear Lagrangian Intrinsic} the values of $\dot q$, $\ddot q$, $\dddot q$ are determined by differentiating \eqref{eq:E-L ODE} sufficiently many times and substituting the initial condition $q(0)$. Note that in case of regular Lagrangians the value of $\dot q$ is determined by the boundary conditions $q(0)$, $q(h)$, and the higher-order derivatives by differentiating the corresponding Euler-Lagrange equations, but apart from that the expression \eqref{eq: Series expansion for exact discrete Lagrangian} remains qualitatively unaffected.

The \emph{midpoint rule} is an integrator obtained by defining the discrete Lagrangian

\begin{equation}
\label{eq: Midpoint rule discrete Lagrangian}
L_d(q,\bar q) = h L\Big( \frac{q+\bar q}{2}, \frac{\bar q-q}{h}\Big).
\end{equation}

\noindent
Calculating the expansion in $h$ yields

\begin{equation}
\label{eq: Series expansion for the midpoint rule discrete Lagrangian}
L_d\big(q(0),q(h) \big) = h L + \frac{h^2}{2} \bigg(\frac{\partial L}{\partial q} \dot q + \frac{\partial L}{\partial \dot q} \ddot q \bigg) + h^3 \bigg(\frac{1}{4}\frac{\partial L}{\partial q} \ddot q + \frac{1}{6} \frac{\partial L}{\partial \dot q} \dddot q + \frac{1}{8}\dot q^T \frac{\partial^2 L}{\partial q^2} \dot q + \frac{1}{4} \dot q^T \frac{\partial^2 L}{\partial q \partial \dot q} \ddot q  + \frac{1}{8}\ddot q^T \frac{\partial^2 L}{\partial \dot q^2} \ddot q\bigg) + o(h^3).
\end{equation}

\noindent
Comparing this to \eqref{eq: Series expansion for exact discrete Lagrangian} shows that the discrete Lagrangian defined by the midpoint rule is second order regardless of the degeneracy of $L$. However, as mentioned before, if $L$ is degenerate we cannot conclude about the global order of convergence of the corresponding discrete Hamiltonian map. The midpoint rule can be formulated as a Runge-Kutta method, namely the 1-stage Gauss method. We discuss Gauss and other Runge-Kutta methods and their convergence properties in more detail in Section~\ref{sec: Variational partitioned Runge-Kutta methods}. Note that low-order variational integrators for Lagrangians \eqref{eq: Linear Lagrangian Intrinsic} based on the midpoint rule have been studied in \cite{RowleyMarsden} and \cite{VankerschaverLeok} in the context of the dynamics of point vortices.

\section{Variational partitioned Runge-Kutta methods}
\label{sec: Variational partitioned Runge-Kutta methods}

\subsection{VPRK methods as PRK methods for the \textquoteleft Hamiltonian' DAE}
\label{sec: VPRK as PRK for the DAE}
To construct higher-order variational integrators one may consider a class of partitioned Runge-Kutta (PRK) methods. Variational partitioned Runge-Kutta (VPRK) methods for regular Lagrangians are described in \cite{HLWGeometric} and \cite{MarsdenWestVarInt}. In this section we show how VPRK methods can be applied to systems described by Lagrangians such as \eqref{eq: Linear Lagrangian Intrinsic}. As in the case of regular Lagrangians, we will construct an $s$-stage variational partitioned Runge-Kutta integrator for the Lagrangian \eqref{eq: Linear Lagrangian Intrinsic} by considering the discrete Lagrangian

\begin{equation}
\label{eq: Discrete Lagrangian for VPRK}
L_d(q,\bar q) = h\sum_{i=1}^{s} b_i L(Q_i,\dot Q_i),
\end{equation}

\noindent
where the internal stages $Q_i$, $\dot Q_i$, $i=1,\ldots,s$, satisfy the relation

\begin{equation}
\label{eq: Internal stages for VPRK}
Q_i = q +  h\sum_{j=1}^{s} a_{ij} \dot Q_j,
\end{equation}

\noindent
and are chosen so that the right-hand side of \eqref{eq: Discrete Lagrangian for VPRK} is extremized under the constraint

\begin{equation}
\label{eq: Internal stages constraint for VPRK}
\bar q = q +  h\sum_{i=1}^{s} b_i \dot Q_i.
\end{equation}

\noindent
A variational integrator is then obtained by forming the corresponding discrete Euler-Lagrange equations \eqref{eq:Discrete Euler-Lagrange equations}.

\begin{thm}
The $s$-stage variational partitioned Runge-Kutta method based on the discrete Lagrangian \eqref{eq: Discrete Lagrangian for VPRK} with the coefficients $a_{ij}$ and $b_i$ is equivalent to the following partitioned Runge-Kutta method applied to the \textquoteleft Hamiltonian' DAE \eqref{eq: Hamiltonian DAE}:

\begin{subequations}
\label{eq: PRK for DAE}
\begin{align}
\label{eq: PRK for DAE 1}
P^i&= \alpha(Q_i), \phantom{_i)]^T \dot Q_i-DH(Q_i),} \qquad \qquad \qquad i=1,\ldots,s,\\
\label{eq: PRK for DAE 2}
\dot P^i &=  [D\alpha(Q_i)]^T \dot Q_i - DH(Q_i), \qquad \qquad \qquad \! i=1,\ldots,s, \\ 
\label{eq: PRK for DAE 3}
Q_i &= q + h \sum_{j=1}^s a_{ij} \dot Q_j, \phantom{-DH(Q_i),} \qquad \qquad \qquad i=1,\ldots,s, \\
\label{eq: PRK for DAE 4}
P^i &= p + h \sum_{j=1}^s \bar a_{ij} \dot P_j,  \phantom{-DH(Q_i),} \qquad \qquad \qquad \, i=1,\ldots,s,\\
\label{eq: PRK for DAE 5}
\bar q &= q + h \sum_{j=1}^s b_j \dot Q_j,\\
\label{eq: PRK for DAE 6}
\bar p &= p + h \sum_{j=1}^s b_j \dot P_j,
\end{align}
\end{subequations}

\noindent
where the coefficients satisfy the condition

\begin{equation}
\label{eq: Symplecticity condition for VPRK}
b_i \bar a_{ij} + b_j a_{ji} = b_i b_j, \quad \qquad \forall i,j=1,\ldots,s,
\end{equation}

\noindent
and $(q,p)$ denote the current values of position and momentum, $(\bar q,\bar p)$ denote the respective values at the next time step, $D\alpha = (\partial\alpha_\mu / \partial q^\nu)_{\mu,\nu=1,\ldots,n}$, $DH = (\partial H / \partial q^\mu)_{\mu=1,\ldots,n}$, and $Q_i$, $\dot Q_i$, $P^i$, $\dot P^i$ are the internal stages, with $Q_i = (Q^\mu_i)_{\mu=1,\ldots,n}$, and similarly for the others.
\end{thm}

\begin{proof}
See Theorem~VI.6.4 in \cite{HLWGeometric} or Theorem~2.6.1 in \cite{MarsdenWestVarInt}. The proof is essentially identical. The only qualitative difference is the fact that in our case the Lagrangian \eqref{eq: Linear Lagrangian Intrinsic} is degenerate, so the corresponding Hamiltonian system is in fact the index-1 differential-algebraic system \eqref{eq: Hamiltonian DAE} rather than a typical system of ordinary differential equations.\\
\end{proof}

\noindent
\paragraph{Existence and uniqueness of the numerical solution.} Given $q$ and $p$, one can use Equations~\eqref{eq: PRK for DAE} to compute the new position $\bar q$ and momentum $\bar p$. First, one needs to solve \eqref{eq: PRK for DAE 1}-\eqref{eq: PRK for DAE 4} for the internal stages $Q_i$, $\dot Q_i$, $P^i$, and $\dot P^i$. This is a system of $4sn$ equations for $4sn$ variables, but one has to make sure these equations are independent, so that a unique solution exists. One may be tempted to calculate the Jacobian of this system for $h=0$, and then use the Implicit Function Theorem. However, even if we start with consistent initial values $(q_0,p_0)$, the numerical solution $(q_k,p_k)$ for $k>0$ will only approximately satisfy the algebraic constraint; so $Q_i=q$ and $P^i=p$ cannot be assumed to be the solution of \eqref{eq: PRK for DAE 1}-\eqref{eq: PRK for DAE 4} for $h=0$, and consequently, the Implicit Function Theorem will not yield a useful result. Let us therefore regard $q$ and $p$ as $h$-dependent, as they result from the previous iterations of the method with the timestep $h$. If the method is convergent, it is reasonable to expect that $p-\alpha(q)$ is small and converges to zero as $h$ is refined. The following approach was inspired by Theorem~4.1 in \cite{HLLectureNotes}.

\begin{thm}
\label{thm: Existence of the numerical solution for PRK}
Let $H$ and $\alpha$ be smooth in an $h$-independent neighborhood $U$ of $q$ and let the matrix 

\begin{equation}
\label{eq: W matrix definition}
W(\xi_1,\ldots,\xi_s)=(\mathcal{\bar{A}}\otimes I_n) \{D\alpha^T\}-(\mathcal{A}\otimes I_n) \{D\alpha\}
\end{equation}

\noindent
be invertible with the inverse bounded in $U^s$, i.e., there exists $C>0$ such that

\begin{equation}
\label{eq: Invertibility and boundedness assumption for PRK}
\big\| W^{-1}(\xi_1,\ldots,\xi_s) \big\| \leq C, \qquad\qquad \forall (\xi_1,\ldots,\xi_s) \in U^s,
\end{equation}

\noindent
where $\mathcal{A}=(a_{ij})_{i,j=1,\ldots,s}$, $\mathcal{\bar A}=(\bar{a}_{ij})_{i,j=1,\ldots,s}$, $I_n$ is the $n\times n$ identity matrix, and $\{D\alpha\}$ denotes the block diagonal matrix

\begin{equation}
\label{eq: Block diagonal Da in terms of xi}
\{D\alpha \}(\xi_1,\ldots,\xi_s)=\bigoplus_{i=1}^s D\alpha(\xi_i) = \textrm{\emph{blockdiag}}\,\big(D\alpha(\xi_1),\ldots,D\alpha(\xi_s) \big).
\end{equation}

\noindent
Suppose also that $(q,p)$ satisfy

\begin{equation}
\label{eq: O(h) assumption for PRK}
p-\alpha(q) = O(h).
\end{equation}
 
\noindent
Then there exists $\bar h>0$ such that the nonlinear system \eqref{eq: PRK for DAE 1}-\eqref{eq: PRK for DAE 4} has a solution for $h \leq \bar h$. The solution is locally unique and satisfies

\begin{equation}
\label{eq: Estimates hypothesis in PRK}
Q_i-q=O(h), \quad\quad P^i-p=O(h), \quad\quad \dot Q_i = O(1), \quad\quad \dot P^i=O(1).
\end{equation}
\end{thm}

\begin{proof}
Substitute \eqref{eq: PRK for DAE 3} and \eqref{eq: PRK for DAE 4} in \eqref{eq: PRK for DAE 1} and \eqref{eq: PRK for DAE 2} to obtain

\begin{align}
\label{eq: System for dot Qi and dot Pi}
0 &= \alpha(Q_i) - p  - h \sum_{j=1}^s \bar a_{ij} \dot P^j, \nonumber \\
\dot P^i &= D\alpha^T(Q_i) \dot Q_i - DH(Q_i),
\end{align}

\noindent
for $i=1,\ldots,s$, where for notational convenience we left the $Q_i$'s as arguments of $\alpha$, $D\alpha^T$ and $DH$, but we keep in mind they are defined by \eqref{eq: PRK for DAE 3}, so that \eqref{eq: System for dot Qi and dot Pi} is a nonlinear system for $\dot Q_i$ and $\dot P^i$. Let us consider the homotopy

\begin{align}
\label{eq: Homotopy for the system for dot Qi and dot Pi}
0 &= \alpha(Q_i) - p  - h \sum_{j=1}^s \bar a_{ij} \dot P^j - (\tau-1) \big( p-\alpha(q) \big), \nonumber \\
\dot P^i &= D\alpha^T(Q_i) \dot Q_i - DH(Q_i) - (\tau -1) DH(q),
\end{align}

\noindent
for $i=1,\ldots,s$. It is easy to see that for $\tau=0$ the system \eqref{eq: Homotopy for the system for dot Qi and dot Pi} has the solution $\dot Q_i = 0$ and $\dot P^i=0$, and for $\tau=1$ it is equivalent to \eqref{eq: System for dot Qi and dot Pi}. Let us treat $\dot Q_i$ and $\dot P^i$ as functions of $\tau$, and differentiate \eqref{eq: Homotopy for the system for dot Qi and dot Pi} with respect to this parameter. The resulting ODE system can be written as

\begin{subequations}
\label{eq: ODE system for dot Qi and dot Pi}
\begin{align}
\label{eq: ODE system for dot Qi and dot Pi 1}
\{D\alpha\} (\mathcal{A}\otimes I_n) \frac{d \dot Q}{d\tau} - \mathcal{\bar{A}}\otimes I_n \frac{d \dot P}{d\tau} = \frac{1}{h}\mathbbm{1}_s\otimes\big(p-\alpha(q)\big), \\
\label{eq: ODE system for dot Qi and dot Pi 2}
\frac{d\dot P}{d\tau} = \Big(\{ D\alpha^T\} + h \{B\}(\mathcal{A}\otimes I_n)\Big)\frac{d \dot Q}{d\tau} - \mathbbm{1}_s\otimes DH(q),
\end{align}
\end{subequations}

\noindent
where for compactness we introduced the following notations: $\dot Q = (\dot Q_1, \ldots, \dot Q_s)^T$, similarly for $\dot P$; $\mathbbm{1}_s = (1,\ldots,1)^T$ is the $s$-dimensional vector of ones; $\{D\alpha \}=\{D\alpha \}(Q_1,\ldots,Q_s)$, and similarly, $\{B\}$ denotes the block diagonal matrix

\begin{equation}
\label{eq: Block diagonal B}
\{B\} = \text{blockdiag}\,\big(B(Q_1,\dot Q_1),\ldots,B(Q_s,\dot Q_s) \big)
\end{equation}

\noindent
with $B(Q_i,\dot Q_i)=D^2\alpha_\beta(Q_i) \dot Q^\beta_i - D^2 H(Q_i)$, where $D^2$ denotes the Hessian matrix of the respective function, and summation over $\beta$ is implied. The system \eqref{eq: ODE system for dot Qi and dot Pi} is further simplified if we substitute \eqref{eq: ODE system for dot Qi and dot Pi 2} in \eqref{eq: ODE system for dot Qi and dot Pi 1}. This way we obtain an ODE system for the variables $\dot Q$ of the form

\begin{align}
\label{eq: ODE system for dot Qi only}
\Big[ (\mathcal{\bar{A}}\otimes I_n)\{D\alpha^T\} - \{D\alpha\} (\mathcal{A}\otimes I_n) + h(\mathcal{\bar{A}}&\otimes I_n)\{B\}(\mathcal{A}\otimes I_n) \Big] \frac{d \dot Q}{d\tau} = \nonumber \\
&(\mathcal{\bar{A}} \mathbbm{1}_s)\otimes DH(q) - \frac{1}{h}\mathbbm{1}_s\otimes\big(p-\alpha(q)\big).
\end{align}

\noindent
Since $\alpha$ is smooth, we have

\begin{equation}
\label{eq: Switching the order of matrix multiplication}
\Big[ \{D\alpha\} (\mathcal{A}\otimes I_n) \Big]_{ij} = a_{ij} D\alpha(Q_i) = a_{ij} D\alpha(Q_j) + O(\delta) = \Big[ (\mathcal{A}\otimes I_n) \{D\alpha\} \Big]_{ij} + O(\delta),
\end{equation}

\noindent
where $\| Q_i-Q_j \| \leq \delta$ for $\delta$ assumed small, but independent of $h$. Moreover, since $\alpha$ and $H$ are smooth, the term $\{B\}$, as a function of $\dot Q$, is bounded in a neighborhood of 0. Therefore, we can write \eqref{eq: ODE system for dot Qi only} as 

\begin{align}
\label{eq: ODE system for dot Qi only using W}
\Big[ W(Q_1,\ldots,Q_s) + O(\delta) + O(h) \Big] \frac{d \dot Q}{d\tau} = (\mathcal{\bar{A}} \mathbbm{1}_s)\otimes DH(q) - \frac{1}{h}\mathbbm{1}_s\otimes\big(p-\alpha(q)\big).
\end{align}

\noindent
By \eqref{eq: Invertibility and boundedness assumption for PRK}, for sufficiently small $h$ and $\delta$, the matrix $W(Q_1,\ldots,Q_s) + O(\delta) + O(h)$ has a bounded inverse, provided that $Q_1,\ldots,Q_s$ remain in $U$. Therefore, the ODE \eqref{eq: ODE system for dot Qi only using W} with the initial condition $\dot Q(0) = 0$ has a unique solution $\dot Q (\tau)$ on a non-empty interval $[0,\bar \tau)$, which can be extended until any of the corresponding $Q_i(\tau)$ leaves $U$. Let us argue that for a sufficiently small $h$ we have $\bar \tau>1$. Given \eqref{eq: Invertibility and boundedness assumption for PRK} and \eqref{eq: O(h) assumption for PRK}, the ODE \eqref{eq: ODE system for dot Qi only using W} implies that

\begin{equation}
\label{eq: d/dtau dotQ is O(1)}
\frac{d \dot Q}{d \tau} = O(1).
\end{equation}

\noindent
Therefore, we have

\begin{equation}
\label{eq: dotQ is O(tau)}
\dot Q(\tau) = \int_0^\tau \frac{d \dot Q}{d \zeta} \, d\zeta = O(\tau)
\end{equation}

\noindent
and further

\begin{equation}
\label{eq: Q is O(tau h)}
Q_i(\tau) = q + O(\tau h)
\end{equation}

\noindent
for $\tau < \bar \tau$. This implies that all $Q_i(\tau)$ remain in $U$ for $\tau \leq 1$ if $h$ is sufficiently small. Consequently, the ODE \eqref{eq: ODE system for dot Qi only} has a solution on the interval $[0,1]$. Then $\dot Q_i(1)$ and $Q_i(1)$ satisfy the estimates \eqref{eq: Estimates hypothesis in PRK}, and are a solution to the nonlinear system \eqref{eq: PRK for DAE 1}-\eqref{eq: PRK for DAE 4}. The corresponding $\dot P^i$ and $P^i$ can be computed using \eqref{eq: PRK for DAE 2} and \eqref{eq: PRK for DAE 4}, and the remaining estimates \eqref{eq: Estimates hypothesis in PRK} can be proved using the fact that $\alpha$ and $H$ are smooth. This completes the proof of the existence of a numerical solution to \eqref{eq: PRK for DAE 1}-\eqref{eq: PRK for DAE 4}.

In order to prove local uniqueness, we substitute the second equation of \eqref{eq: System for dot Qi and dot Pi} in the first one to obtain a nonlinear system for $\dot Q_i$, namely

\begin{align}
\label{eq: System for dot Qi}
0 &= \alpha(Q_i) - p  - h \sum_{j=1}^s \bar a_{ij}\big(D\alpha^T(Q_j) \dot Q_j - DH(Q_j) \big),
\end{align} 

\noindent
for $i=1,\ldots,s$, where we again left the $Q_i$'s for notational convenience. Suppose there exists another solution $\dot{ \bar {Q}}_i$ that satisfies the estimates \eqref{eq: Estimates hypothesis in PRK}, and denote $\Delta \dot Q_i = \dot{ \bar {Q}}_i-\dot Q_i$. Based on the assumptions, we have $\Delta \dot Q_i = O(1)$, i.e., it is at least bounded as $h \longrightarrow 0$. We will show that for sufficiently small $h$ we in fact have $\Delta \dot Q_i =0$. Since $\dot{\bar{Q}}_i$ satisfy \eqref{eq: System for dot Qi}, we have

\begin{align}
\label{eq: System for alternate dot Qi}
0 &= \alpha(\bar Q_i) - p  - h \sum_{j=1}^s \bar a_{ij}\big(D\alpha^T(\bar Q_j) \dot {\bar{Q}}_j - DH(\bar Q_j) \big)
\end{align} 

\noindent
for $i=1,\ldots,s$. Subtract \eqref{eq: System for dot Qi} from \eqref{eq: System for alternate dot Qi}, and linearize around $\dot Q_i$. Based on the fact that $\Delta \dot Q_i = O(1)$, and using the notation introduced before, we get

\begin{equation}
\label{eq: Linearization for local uniqueness}
0=h \Big[ \{D\alpha\} (\mathcal{A}\otimes I_n) - (\mathcal{\bar{A}}\otimes I_n)\{D\alpha^T\} \Big] \Delta \dot Q + O(h^2 \|\Delta \dot Q \|).
\end{equation} 

\noindent
By a similar argument as before, for sufficiently small $h$ the matrix $\Big[ \{D\alpha\} (\mathcal{A}\otimes I_n) - (\mathcal{\bar{A}}\otimes I_n)\{D\alpha^T\} \Big]$ has a bounded inverse, therefore \eqref{eq: Linearization for local uniqueness} implies $\Delta \dot Q = O(h \|\Delta \dot Q \|)$, that is,

\begin{equation}
\label{eq: Estimate for Delta dot Q}
\|\Delta \dot Q \| \leq \tilde C h \|\Delta \dot Q \| \qquad \Longleftrightarrow \qquad(1-\tilde C h) \|\Delta \dot Q \| \leq 0
\end{equation}

\noindent
for some constant $\tilde C >0$. Note that for $h<1/\tilde C$ we have $(1-\tilde C h)>0$, and therefore $\|\Delta \dot Q \|=0$, which completes the proof of the local uniqueness of a numerical solution to \eqref{eq: PRK for DAE 1}-\eqref{eq: PRK for DAE 4}.

\end{proof}.

\paragraph{Remarks.} The condition \eqref{eq: Invertibility and boundedness assumption for PRK} may be tedious to verify, especially if one uses a Runge-Kutta method with many stages. However, this condition is significantly simplified in the following special cases:

\begin{enumerate}
\item For a non-partitioned Runge-Kutta method we have $\mathcal{A}=\mathcal{\bar{A}}$, and the condition \eqref{eq: Invertibility and boundedness assumption for PRK} is satisfied if $\mathcal{A}$ is invertible, and the mass matrix $M(q) = D\alpha^T(q) - D\alpha(q)$, as defined in Section~\ref{sec: Equations of motion}, is invertible in $U$ and its inverse is bounded.

\item If $D\alpha$ is antisymmetric, then the condition \eqref{eq: Invertibility and boundedness assumption for PRK} is satisfied if $(\mathcal{A}+\mathcal{\bar{A}})$ is invertible, and the matrix $D\alpha(q)$ is invertible in $U$ and its inverse is bounded.
\end{enumerate}

\subsection{Linear $\alpha_\mu(q)$}
\label{sec: Linear alpha}

An interesting special case is obtained if we have, in some local chart on $Q$, $\alpha_\mu(q) = - \frac{1}{2}\Lambda_{\mu \nu} q^\nu$ for some constant matrix $\Lambda$. Without loss of generality assume that $\Lambda$ is invertible and antisymmetric. The Lagrangian \eqref{eq: Linear Lagrangian in Coordinates} then takes the form

\begin{equation}
\label{eq: Bilinear Lagrangian in coordinates}
L(q,\dot q)=-\frac{1}{2}\Lambda_{\mu \nu} \dot q^\mu q^\nu - H(q),
\end{equation} 

\noindent
the Euler-Lagrange equations \eqref{eq:E-L ODE} become

\begin{equation}
\label{eq:E-L ODE for linear alpha}
\Lambda \dot q = DH(q),
\end{equation}

\noindent
and the \textquoteleft Hamiltonian' DAE system \eqref{eq: Hamiltonian DAE} is

\begin{align}
\label{eq: Hamiltonian DAE for linear alpha}
p &= -\frac{1}{2} \Lambda q, \nonumber \\
\dot p &=\frac{1}{2} \Lambda \dot q - DH(q).
\end{align}

\noindent
Let us consider a special case of the method \eqref{eq: PRK for DAE} with $a_{ij}=\bar a_{ij}$, i.e., a non-partitioned Runge-Kutta method. Applying it to \eqref{eq: Hamiltonian DAE for linear alpha} we get

\begin{subequations}
\label{eq: PRK for DAE for linear alpha}
\begin{align}
\label{eq: PRK for DAE for linear alpha 1}
P^i&= -\frac{1}{2} \Lambda Q_i, \quad \qquad \qquad \qquad \qquad \!\! i=1,\ldots,s,\\
\label{eq: PRK for DAE for linear alpha 2}
\dot P^i &=  \frac{1}{2} \Lambda \dot Q_i - DH(Q_i), \qquad \qquad \:\:\, i=1,\ldots,s, \\ 
\label{eq: PRK for DAE for linear alpha 3}
Q_i &= q + h \sum_{j=1}^s a_{ij} \dot Q_j,  \qquad \qquad \qquad i=1,\ldots,s, \\
\label{eq: PRK for DAE for linear alpha 4}
P^i &= p + h \sum_{j=1}^s a_{ij} \dot P_j,   \qquad \qquad \qquad \, i=1,\ldots,s,\\
\label{eq: PRK for DAE for linear alpha 5}
\bar q &= q + h \sum_{j=1}^s b_j \dot Q_j,\\
\label{eq: PRK for DAE for linear alpha 6}
\bar p &= p + h \sum_{j=1}^s b_j \dot P_j.
\end{align}
\end{subequations}

\noindent
Since $\Lambda$ is antisymmetric and invertible, then by Theorem~\ref{thm: Existence of the numerical solution for PRK} the scheme \eqref{eq: PRK for DAE for linear alpha} yields a unique numerical solution to \eqref{eq: Hamiltonian DAE for linear alpha} if the Runge-Kutta matrix $\mathcal{A}=(a_{ij})$ is invertible.

\begin{thm}
\label{thm: Equivalence with a RK for the ODE}
Suppose $\mathcal{A}=(a_{ij})$ is invertible and $p=-\frac{1}{2}\Lambda q$. Then the method \eqref{eq: PRK for DAE for linear alpha} is equivalent to the same Runge-Kutta method applied to \eqref{eq:E-L ODE for linear alpha}.
\end{thm}

\begin{proof}
Substitute \eqref{eq: PRK for DAE for linear alpha 3} and \eqref{eq: PRK for DAE for linear alpha 4} in \eqref{eq: PRK for DAE for linear alpha 1}, and use the fact $p=-\frac{1}{2}\Lambda q$ to obtain

\begin{equation}
\sum_{j=1}^s a_{ij} \Big(\dot P_j + \frac{1}{2}\Lambda \dot Q_j\Big)=0, \qquad \qquad \qquad \, i=1,\ldots,s.
\end{equation}

\noindent
Since $\mathcal{A}$ is invertible, this implies

\begin{equation}
\label{eq: Pi vs Qi for linear alpha}
\dot P_i =- \frac{1}{2}\Lambda \dot Q_i, \qquad \qquad \qquad \, i=1,\ldots,s.
\end{equation}

\noindent
Substituting this in \eqref{eq: PRK for DAE for linear alpha 2} yields

\begin{equation}
\Lambda \dot Q_i = DH(Q_i), \qquad \qquad \qquad \, i=1,\ldots,s.
\end{equation}

\noindent
Together with \eqref{eq: PRK for DAE for linear alpha 3} and \eqref{eq: PRK for DAE for linear alpha 5}, this gives a Runge-Kutta method for \eqref{eq:E-L ODE for linear alpha}. Moreover, substituting \eqref{eq: Pi vs Qi for linear alpha} and $p=-\frac{1}{2}\Lambda q$ in \eqref{eq: PRK for DAE for linear alpha 6}, and using \eqref{eq: PRK for DAE for linear alpha 5}, one has

\begin{equation}
\bar p = -\frac{1}{2}\Lambda q + h \sum_{j=1}^s b_j \Big( -\frac{1}{2}\Lambda \dot Q_j \Big) = -\frac{1}{2}\Lambda \bar q,
\end{equation}

\noindent
that is, $(\bar q, \bar p)$ satisfy the algebraic constraint.\\
\end{proof}

\begin{corollary}
\label{thm: Invariance of the primary constraint}
The numerical flow on $T^*Q$ defined by \eqref{eq: PRK for DAE for linear alpha} leaves the primary constraint $N$ invariant, i.e., if $(q,p)\in N$, then $(\bar q,\bar p)\in N$.
\end{corollary}

\noindent
If the coefficients of the method \eqref{eq: PRK for DAE for linear alpha} satisfy the condition \eqref{eq: Symplecticity condition for VPRK}, then \eqref{eq: PRK for DAE for linear alpha} is a variational integrator and the associated discrete Hamiltonian map $\tilde F_{L_d}$ is symplectic on $T^*Q$, as explained in Section~\ref{sec: Discrete Mechanics}. Given Corollary~\ref{thm: Invariance of the primary constraint}, we further have:

\begin{corollary}
\label{thm: Symplecticity on the primary constraint}
If the coefficients $a_{ij}$ and $b_i$ in \eqref{eq: PRK for DAE for linear alpha} satisfy the condition \eqref{eq: Symplecticity condition for VPRK}, then the discrete Hamiltonian map $\tilde F_{L_d}$ associated with \eqref{eq: Discrete Lagrangian for VPRK} is symplectic on the primary constraint $N$, that is, $\big(\tilde F_{L_d} \big|_N \big)^* \tilde \Omega_N = \tilde \Omega_N$.
\end{corollary}

\paragraph{Convergence.} Various Runge-Kutta methods and their classical orders of convergence, that is, orders of convergence when applied to (non-stiff) ordinary differential equations, are discussed in many textbooks on numerical analysis, for instance \cite{HWODE1} and \cite{HWODE2}. When applied to differential-algebraic equations, the order of convergence of a Runge-Kutta method may be reduced (see \cite{PetzoldDAE}, \cite{HWODE2}, \cite{Rabier6}). However, in the case of \eqref{eq: Hamiltonian DAE for linear alpha} Theorem~\ref{thm: Equivalence with a RK for the ODE} implies that the classical order of convergence of non-partitioned Runge-Kutta methods \eqref{eq: PRK for DAE for linear alpha} is retained. 

\begin{thm}
\label{thm: Retention of the classical order of convergence}
A Runge-Kutta method with the coefficients $a_{ij}$ and $b_i$ applied to the DAE system \eqref{eq: Hamiltonian DAE for linear alpha} retains its classical order of convergence.
\end{thm}

\begin{proof}
Let $r$ be the classical order of the considered Runge-Kutta method, $(q,p)\in N$ an initial condition, $(q_E(t),p_E(t))$ the exact solution to \eqref{eq: Hamiltonian DAE for linear alpha} such that $(q_E(0),p_E(0))=(q,p)$, and $(q_k,p_k)$ the numerical solution obtained by applying the method \eqref{eq: PRK for DAE for linear alpha} iteratively $k$ times with $(q_0,p_0)=(q,p)$. Theorem~\ref{thm: Equivalence with a RK for the ODE} states that the method \eqref{eq: PRK for DAE for linear alpha} is equivalent to applying the same Runge-Kutta method to the ODE system \eqref{eq:E-L ODE for linear alpha}. Hence, we obtain convergence of order $r$ in the $q$ variable, that is, for a fixed time $T>0$ and an integer $K$ such that $h=T/K$, we have the estimate

\begin{equation}
\label{eq: Convergence in the q component}
\|q_K-q(T)\| \leq C h^{r+1}
\end{equation}

\noindent
for some constant $C>0$ (cf. Definition~\ref{thm: Definition of the order of convergence}). By Corollary~\ref{thm: Invariance of the primary constraint} we know that $p_K = -\frac{1}{2}\Lambda q_K$, so we have the estimate

\begin{equation}
\label{eq: Convergence in the p component}
\|p_K-p(T)\|\leq \frac{1}{2} \|\Lambda\| \|q_K-q(T)\| \leq \frac{1}{2} \|\Lambda\| C h^{r+1},
\end{equation}

\noindent
which completes the proof, since $\|\Lambda\| < +\infty$.\\
\end{proof}

\noindent
Of particular interest to us are Runge-Kutta methods that satisfy the condition \eqref{eq: Symplecticity condition for VPRK}, for instance symplectic diagonally-implicit Runge-Kutta methods (DIRK) or Gauss collocation methods (see \cite{HLWGeometric}). The $s$-stage Gauss method is of classical order $2s$, therefore we have:

\begin{corollary}
\label{thm: Order of convergence of Gauss methods}
The $s$-stage Gauss collocation method applied to the DAE system \eqref{eq: Hamiltonian DAE for linear alpha} is convergent of order $2s$.
\end{corollary}

\noindent
As mentioned in Section~\ref{sec: Variational error analysis}, the midpoint rule is a 1-stage Gauss method, therefore it retains its classical second order of convergence.

\paragraph{Backward error analysis.} The system \eqref{eq:E-L ODE for linear alpha} can be rewritten as the Poisson system

\begin{equation}
\label{eq:Poisson E-L ODE for linear alpha}
\dot q = \Lambda^{-1}DH(q)
\end{equation}

\noindent
with the structure matrix $\Lambda^{-1}$ (see \cite{MarsdenRatiuSymmetry}, \cite{HLWGeometric}). The flow $\varphi_t$ for this equation is a \emph{Poisson map}, that is, it satisfies the property

\begin{equation}
\label{eq: Poisson map property}
D\varphi_t(q) \, \Lambda^{-1} \,[D\varphi_t(q)]^{T} = \Lambda^{-1},
\end{equation}

\noindent
which is in fact equivalent to the symplecticity property \eqref{eq: phi symplectic on Q} or \eqref{eq: tilde phi symplectic on N} written in local coordinates on $Q$ or $N$, respectively. Let $F_h: Q \longrightarrow Q$ represent the numerical flow defined by some numerical algorithm applied to \eqref{eq:Poisson E-L ODE for linear alpha}. We say this flow is a \emph{Poisson integrator} if

\begin{equation}
\label{eq: Poisson integrator criterion}
DF_h(q) \, \Lambda^{-1} \,[DF_h(q)]^{T} = \Lambda^{-1}.
\end{equation}

\noindent
The left-hand side of \eqref{eq: Poisson map property} can be regarded as a quadratic invariant of \eqref{eq:Poisson E-L ODE for linear alpha}. By Theorem~\ref{thm: Equivalence with a RK for the ODE} the method \eqref{eq: PRK for DAE for linear alpha} is equivalent to applying the same Runge-Kutta method to \eqref{eq:Poisson E-L ODE for linear alpha}. If its coefficients also satisfy the condition \eqref{eq: Symplecticity condition for VPRK}, then it can be shown that the method preserves quadratic invariants (see Theorem~IV.2.2 in \cite{HLWGeometric}). Therefore, we have:

\begin{corollary}
\label{thm: R-K methods as Poisson integrators}
If $\mathcal{A}=(a_{ij})$ is invertible, the coefficients $a_{ij}$ and $b_i$ satisfy the condition \eqref{eq: Symplecticity condition for VPRK}, and $p=-\frac{1}{2}\Lambda q$, then the method \eqref{eq: PRK for DAE for linear alpha} is a Poisson integrator for \eqref{eq:Poisson E-L ODE for linear alpha}.
\end{corollary}

\noindent
The true power of symplectic integrators for Hamiltonian equations is revealed through their backward error analysis: a symplectic integrator for a Hamiltonian system with the Hamiltonian $H(q,p)$ defines the \emph{exact} flow for a nearby Hamiltonian system, whose Hamiltonian can be expressed as the asymptotic series

\begin{equation}
\label{eq:ModifiedHamiltonian}
\tilde{H}(q,p) = H(q,p) + h H_2(q,p) + h^2 H_3(q,p) + \ldots
\end{equation}

\noindent
Owing to this fact, under some additional assumptions, symplectic numerical schemes nearly conserve the original Hamiltonian $H(q,p)$ over exponentially long time intervals (see \cite{HLWGeometric} for details). A similar result holds for Poisson integrators for Poisson systems: a Poisson integrator defines the exact flow for a nearby Poisson system, whose structure matrix is the same and whose Hamiltonian has the asymptotic expansion \eqref{eq:ModifiedHamiltonian} (see Theorem~IX.3.6 in \cite{HLWGeometric}). Therefore, we expect the non-partitioned Runge-Kutta schemes \eqref{eq: PRK for DAE for linear alpha} satisfying the condition \eqref{eq: Symplecticity condition for VPRK} to demonstrate good preservation of the original Hamiltonian $H$. See Section~\ref{sec: Numerical experiments} for numerical examples.

Partitioned Runge-Kutta methods do not seem to have special properties when applied to systems with linear $\alpha_\mu(q)$, therefore we describe them in the general case in Section~\ref{sec: Nonlinear alpha}.

\subsection{Nonlinear $\alpha_\mu(q)$}
\label{sec: Nonlinear alpha}

When the coordinates $\alpha_\mu(q)$ are nonlinear functions of $q$, then the Runge-Kutta methods discussed in Section~\ref{sec: Linear alpha} lose some of their properties: a theorem similar to Theorem~\ref{thm: Equivalence with a RK for the ODE} cannot be proved, most of the Runge-Kutta methods (whether non-partitioned or partitioned) do not preserve the algebraic constraint $p = \alpha(q)$, i.e., the numerical solution does not stay on the primary constraint $N$, and therefore their order of convergence is reduced, unless they are \emph{stiffly accurate}.

\subsubsection{Runge-Kutta methods}
\label{sec: Runge-Kutta methods}

Let us again consider non-partitioned methods with $a_{ij}=\bar a_{ij}$. Convergence results for some classical Runge-Kutta schemes of interest can be obtained by transforming \eqref{eq: Hamiltonian DAE} into a semi-explicit index-2 DAE system. Let us briefly review this approach. More details can be found in \cite{HLLectureNotes} and \cite{HWODE2}.

The system \eqref{eq: Hamiltonian DAE} can be written as the quasi-linear DAE

\begin{equation}
\label{eq: Quasi-linear DAE form}
C(y) \dot y = f(y),
\end{equation}

\noindent
where $y=(q,p)$ and

\begin{equation}
\label{eq: C(y) and f(y)}
C(y) = \left(
\begin{matrix}
[D\alpha(q)]^T & -I_n \\
0 & 0 \\
\end{matrix}
\right),
\qquad\qquad\qquad
f(y) = \left(
\begin{matrix}
DH(q) \\
p-\alpha(q) \\
\end{matrix}
\right),
\end{equation}

\noindent
where $I_n$ denotes the $n\times n$ identity matrix. Let us introduce a slack variable $z$ and rewrite \eqref{eq: Quasi-linear DAE form} as the index-2 DAE system

\begin{subequations}
\label{eq: Index 2 DAE form}
\begin{align}
\label{eq: Index 2 DAE form 1}
\dot y &= z, \\
\label{eq: Index 2 DAE form 2}
0 &= C(y)z-f(y).
\end{align}
\end{subequations}

\noindent
This system is of index 2, because it has $4n$ dependent variables, but only $2n$ differential equations \eqref{eq: Index 2 DAE form 1}, and some components of the algebraic equations \eqref{eq: Index 2 DAE form 2} have to be differentiated twice with respect to time in order to derive the missing differential equations for $z$. Note that $C(y)$ is a singular matrix of constant rank $n$, therefore it can be decomposed (using Gauss elimination or the singular value decomposition) as

\begin{equation}
\label{eq: C(y) decomposition}
C(y) = S(y)
\left(
\begin{matrix}
I_n & 0 \\
0 & 0 \\
\end{matrix}
\right)
T(y)
\end{equation}

\noindent
for some non-singular matrices $S(y)$ and $T(y)$. Since $\alpha(q)$ is assumed to be smooth, one can choose $S$ and $T$ so that they are also smooth (at least in a neighborhood of $y$). Premultiplying both sides of \eqref{eq: Index 2 DAE form 2} by $S^{-1}(y)$ turns the DAE \eqref{eq: Index 2 DAE form} into

\begin{subequations}
\label{eq: Index 2 DAE block form}
\begin{align}
\label{eq: Index 2 DAE block form 1}
\dot y_1 &= z_1, \\
\label{eq: Index 2 DAE block form 2}
\dot y_2 &= z_2, \\
\label{eq: Index 2 DAE block form 3}
0 &= T_{11}(y) \,z_1 + T_{12}(y) \, z_2-\tilde f_1(y), \\
\label{eq: Index 2 DAE block form 4}
0 &= \tilde f_2(y),
\end{align}
\end{subequations}

\noindent
where we introduced the block structure $y=(y_1, y_2)$, $z = (z_1,z_2)$, and

\begin{equation}
\label{eq: Block structure of the index 2 DAE}
T(y)=\left(
\begin{matrix}
T_{11} & T_{12} \\
T_{21} & T_{22}
\end{matrix} \right),
\qquad\qquad
S^{-1}(y) \, f(y) = \left(
\begin{matrix}
\tilde f_1(y) \\
\tilde f_2(y)
\end{matrix} \right).
\end{equation}

\noindent
Since $T(y)$ is invertible, we can assume without loss of generality that the block $T_{11}(y)$ is invertible, too (one can always permute the columns of $T(y)$ otherwise). Let us compute $z_1$ from \eqref{eq: Index 2 DAE block form 3} and substitute it in \eqref{eq: Index 2 DAE block form 1}. The resulting system,

\begin{subequations}
\label{eq: Index 2 DAE final block form}
\begin{align}
\label{eq: Index 2 DAE final block form 1}
\dot y_1 &= \big(T_{11}(y)\big)^{-1} \big( \tilde f_1(y) - T_{12}(y) z_2\big), \\
\label{eq: Index 2 DAE final block form 2}
\dot y_2 &= z_2, \\
\label{eq: Index 2 DAE final block form 3}
0 &= \tilde f_2(y),
\end{align}
\end{subequations}

\noindent
has the form of a semi-explicit index-2 DAE

\begin{align}
\label{eq: Generic semi-explicit index 2 DAE}
\dot y &= F(y,z_2), \nonumber \\
0 &= G(y),
\end{align}

\noindent
provided that 

\begin{equation}
\label{eq: Index 2 condition}
D_y G \,D_{z_2} F = -D_{y_1} \tilde f_2 \,T_{11}^{-1}\, T_{12}+ D_{y_2} \tilde f_2
\end{equation}

\noindent
has a bounded inverse.

It is an elementary exercise to show that the partitioned Runge-Kutta method \eqref{eq: PRK for DAE} is invariant under the presented transformation, that is, it defines a numerically equivalent partitioned Runge-Kutta method for \eqref{eq: Index 2 DAE final block form}. Runge-Kutta methods for semi-explicit index-2 DAEs have been studied and some convergence results are available. Convergence estimates for the $y$ component of \eqref{eq: Index 2 DAE final block form} can be readily applied to the solution of \eqref{eq: Quasi-linear DAE form}.

As in Section~\ref{sec: Linear alpha}, of particular interest to us are variational Runge-Kutta methods, i.e., methods satisfying the condition \eqref{eq: Symplecticity condition for VPRK}, for example Gauss collocation methods (see \cite{HLWGeometric}, \cite{HWODE1}). However, in the case when $\alpha(q)$ is a nonlinear function, the solution generated by the Gauss methods does not stay on the primary constraint $N$ and this affects their rate of convergence, as will be shown below. For comparison, we will also consider the Radau IIA methods (see \cite{HWODE2}), which, although not variational/symplectic, are \emph{stiffly accurate}, that is, their coefficients satisfy $a_{sj} = b_j$ for $j=1,\dots,s$, so the numerical value of the solution at the new time step is equal to the value of the last internal stage, and therefore the numerical solution stays on the submanifold $N$. We cite the following convergence rates for the $y$ component of \eqref{eq: Generic semi-explicit index 2 DAE} after \cite{HWODE2} and \cite{HLLectureNotes}:

\begin{itemize}
\item $s$-stage Gauss method---convergent of order $\left\{ \begin{array}{cl} s+1 & \text{for $s$ odd} \\ s & \text{for $s$ even} \end{array}\right.$,
\item $s$-stage Radau IIA method---convergent of order $2s-1$.
\end{itemize}

\noindent
With the exception of the midpoint rule ($s=1$), we see that the order of convergence of the Gauss methods is reduced. On the other hand, the Radau IIA methods retain their classical order $2s-1$.

\paragraph{Symplecticity.} Since the Gauss methods satisfy the condition \eqref{eq: Symplecticity condition for VPRK}, they generate a flow which preserves the canonical symplectic form $\tilde \Omega$ on $T^*Q$, as explained in Section~\ref{sec: Discrete Mechanics}. However, since the primary constraint $N$ is not invariant under this flow, a result analogous to Corollary~\ref{thm: Symplecticity on the primary constraint} does not hold, i.e., the flow is not symplectic on $N$.

\subsubsection{Partitioned Runge-Kutta methods}
\label{sec: Partitioned Runge-Kutta methods}

In Section~\ref{sec: Numerical experiments} we present numerical results for the Lobatto IIIA-IIIB methods (see \cite{HLWGeometric}). Their numerical performance appears rather unattractive, therefore our theoretical results regarding partitioned Runge-Kutta methods are less complete. Below we summarize the experimental orders of convergence of the Lobatto~IIIA-IIIB schemes that we observed in our numerical computations (see Figure~\ref{fig: Convergence plots for Kepler's problem}, Figure~\ref{fig: Convergence plot for point vortices}, and Figure~\ref{fig: Convergence plot for the Lotka-Volterra model}):

\begin{itemize}
\item $2$-stage Lobatto~IIIA-IIIB---inconsistent,
\item $3$-stage Lobatto~IIIA-IIIB---convergent of order 2,
\item $4$-stage Lobatto~IIIA-IIIB---convergent of order 2.
\end{itemize}

Comments regarding the symplecticity of these schemes are the same as for the Gauss methods mentioned above in Section~\ref{sec: Runge-Kutta methods}.

\section{Numerical experiments}
\label{sec: Numerical experiments}

In this section we present the results of the numerical experiments we performed to test the methods discussed in Section~\ref{sec: Variational partitioned Runge-Kutta methods}. We consider Kepler's problem, the dynamics of planar point vortices, and the Lotka-Volterra model, and we show how each of these models can be formulated as a Lagrangian system linear in velocities.

\subsection{Kepler's problem}
\label{sec: Kepler's problem}

A particle or a planet moving in a central potential in two dimensions can be described by the Hamiltonian

\begin{equation}
\label{eq: Kepler's problem Hamiltonian}
H(x,y,p_x,p_y) = \frac{1}{2}p_x^2+\frac{1}{2}p_x^2 -\frac{1}{\sqrt{x^2+y^2}}-H_0,
\end{equation}

\noindent
where $(x,y)$ denotes the position of the planet and $(p_x,p_y)$ its momentum; $H_0$ is an arbitrary constant. The corresponding Lagrangian can be obtained in the usual way as

\begin{equation}
\label{eq: General L for Kepler's problem}
L = p_x \dot x + p_y \dot y - H(x,y,p_x,p_y).
\end{equation}

\noindent
If one performs the standard Legendre transform $\dot x = \partial H / \partial p_x$, $\dot y = \partial H / \partial p_y$, then $L=L(x,y,\dot x, \dot y)$ will take the usual nondegenerate form, quadratic in velocities. However, one can also introduce the variable $q=(x,y,p_x,p_y)$ and view $L=L(q,\dot q)$ as \eqref{eq: Linear Lagrangian in Coordinates}, that is, a Lagrangian linear in velocities (see~\cite{Faddeev}). Comparing \eqref{eq: General L for Kepler's problem} and \eqref{eq: Bilinear Lagrangian in coordinates} we see that the corresponding $\Lambda$ is singular. Without loss of generality we replace $\Lambda$ with its antisymmetric part $(\Lambda - \Lambda^T)/2$, which is invertible, and consider the Lagrangian

\begin{equation}
\label{eq: L for Kepler's problem}
L = \frac{1}{2} q^3 \dot q^1 + \frac{1}{2} q^4 \dot q^2 - \frac{1}{2} q^1 \dot q^3 - \frac{1}{2} q^2 \dot q^4 - H(q).
\end{equation}

As a test problem we considered an elliptic orbit with eccentricity $e=0.5$ and semi-major axis $a=1$. We took the initial condition at the pericenter, i.e., $q^1_{init}=(1-e)a=0.5$, $q^2_{init}=0$, $q^3_{init} = 0$, $q^4_{init} = a\sqrt{(1+e)/(1-e)}\approx 1.73$. This is a periodic orbit with period $T_{period}=2 \pi$. A reference solution was computed by integrating \eqref{eq:E-L ODE for linear alpha} until the time $T=7$ using Verner's method (a 6-th order explicit Runge-Kutta method; see \cite{HWODE1}) with the small time step $h = 2 \times 10^{-7}$. The reference solution is depicted in Figure~\ref{fig: Reference solution for Kepler's problem}.

\begin{figure}[tbp]
	\centering
		\includegraphics[width=0.8\textwidth]{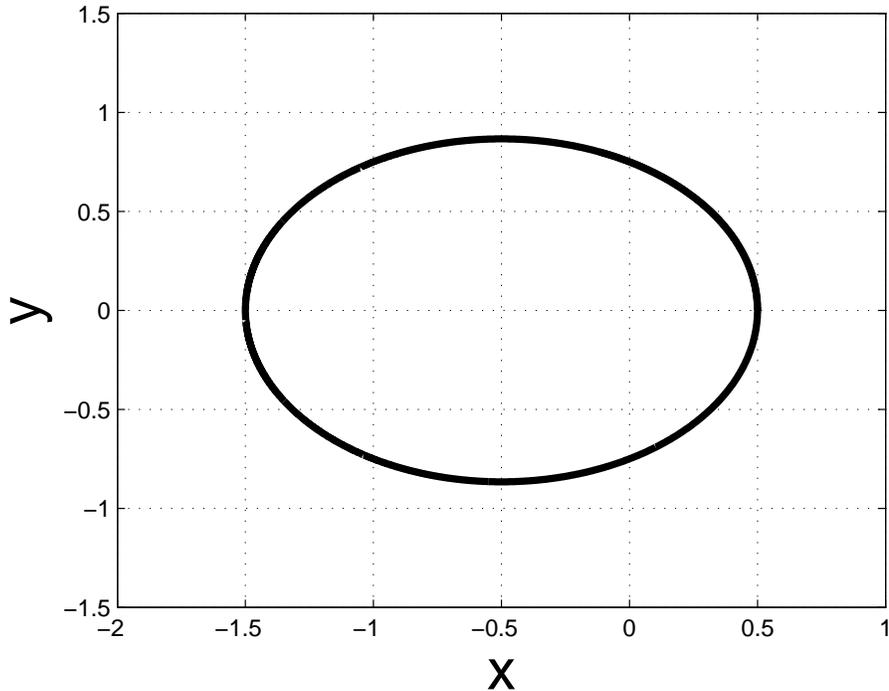}
		\caption{The reference solution for Kepler's problem computed by integrating \eqref{eq:E-L ODE for linear alpha} until the time $T=7$ using Verner's method with the time step $h = 2 \times 10^{-7}$.}
		\label{fig: Reference solution for Kepler's problem}
\end{figure}

We solved the same problem using several of the methods discussed in Section~\ref{sec: Variational partitioned Runge-Kutta methods} for a number of time steps ranging from $h=3.5\times10^{-3}$ to $h=3.5\times10^{-1}$. The value of the solutions at $T=7$ was then compared against the reference solution. The max norm errors are depicted in Figure~\ref{fig: Convergence plots for Kepler's problem}. We see that the rates of convergence of the Gauss and the 3-stage Radau IIA methods are consistent with Theorem~\ref{thm: Retention of the classical order of convergence} and Corollary~\ref{thm: Order of convergence of Gauss methods}. For the Lobatto IIIA-IIIB methods we observe a reduction of order. The 2-stage Lobatto IIIA-IIIB method turns out to be inconsistent and is not depicted in Figure~\ref{fig: Convergence plots for Kepler's problem}. Both the 3- and 4-stage methods converge only quadratically, while their classical orders of convergence are 4 and 6, respectively.

\begin{figure}[tbp]
	\centering
		\includegraphics[width=0.8\textwidth]{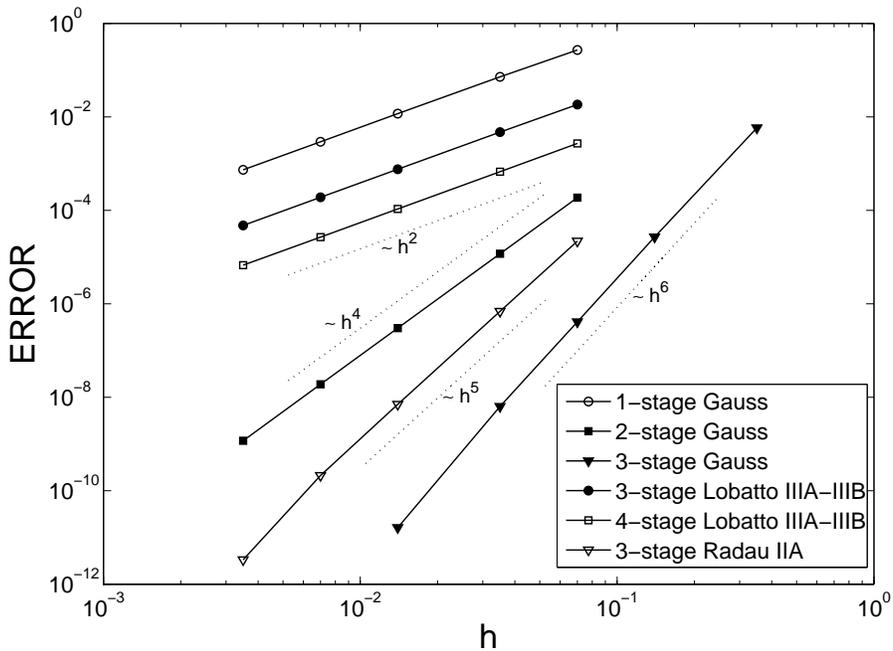}
		\caption{Convergence of several Runge-Kutta methods for Kepler's problem.}
		\label{fig: Convergence plots for Kepler's problem}
\end{figure}

We also investigated the long-time behavior of our integrators and conservation of the Hamiltonian. For convenience, we set $H_0=-0.5$ in \eqref{eq: Kepler's problem Hamiltonian}, so that $H=0$ on the considered orbit. We applied the Gauss methods with the relatively large time step $h=0.1$ and computed the numerical solution until the time $T=5\times 10^5$. Figure~\ref{fig: Energy plots for Gauss methods for Kepler's problem} shows that the Gauss integrators preserve the Hamiltonian very well, which is consistent with Corollary~\ref{thm: R-K methods as Poisson integrators}. We performed similar computations for the Lobatto IIIA-IIIB and Radau IIA methods, also with $h=0.1$. The results are depicted in Figure~\ref{fig: Energy plots for Lobatto and Radau methods for Kepler's problem}. The 3- and 4-stage Lobatto IIIA-IIIB schemes result in instabilities, the planet's trajectory spirals down on the center of gravity, and the computations cannot be continued too far in time. The Hamiltonian shows major variations whose amplitude grows in time. The non-variational Radau IIA scheme yields an accurate solution, but it demonstrates a gradual energy dissipation.

\begin{figure}[tbp]
	\centering
		\includegraphics[width=0.9\textwidth]{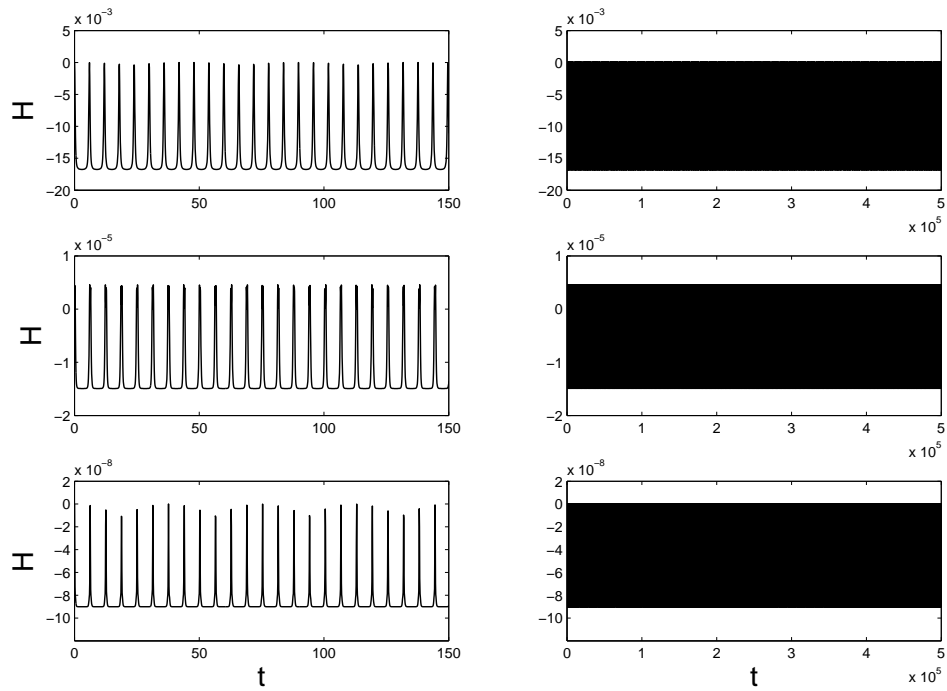}
		\caption{Hamiltonian conservation for the 1-stage (\emph{top row}), 2-stage (\emph{middle row}) and 3-stage (\emph{bottom row}) Gauss methods applied to Kepler's problem with the time step $h=0.1$ over the time interval $[0,5\times 10^5]$ (\emph{right column}), with a close-up on the initial interval $[0,150]$ shown in the \emph{left column}.}
		\label{fig: Energy plots for Gauss methods for Kepler's problem}
\end{figure}

\begin{figure}[tbp]
	\centering
		\includegraphics[width=0.9\textwidth]{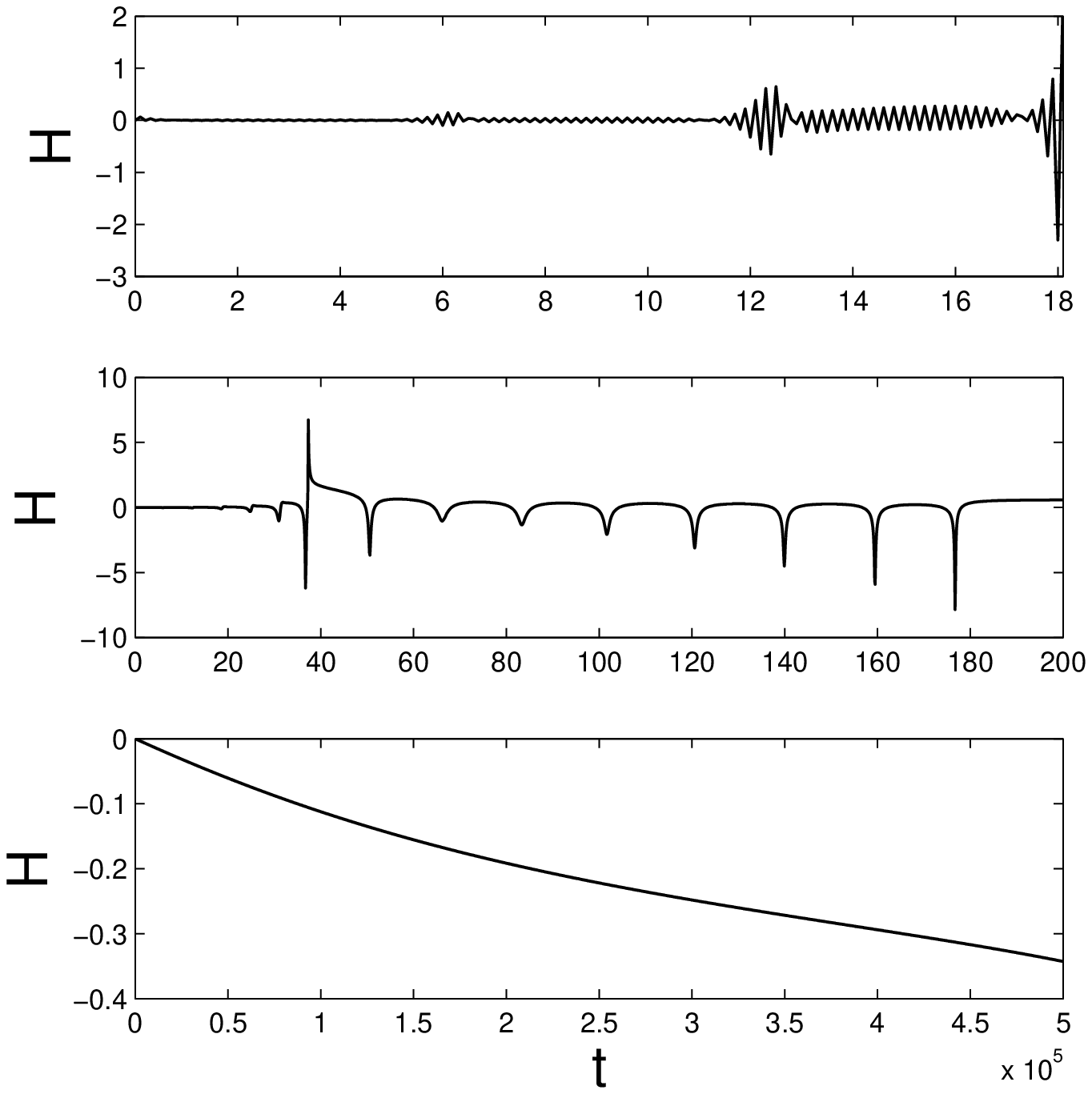}
		\caption{Hamiltonian for the numerical solution of Kepler's problem obtained with the 3- and 4-stage Lobatto IIIA-IIIB schemes (\emph{top} and \emph{middle}, respectively), and the non-variational Radau IIA method (\emph{bottom}).}
		\label{fig: Energy plots for Lobatto and Radau methods for Kepler's problem}
\end{figure}

\subsection{Point vortices}
\label{sec: Point vortices}

Point vortices in the plane are another interesting example of a system with linear $\alpha_\mu(q)$ (see \cite{Newton}, \cite{RowleyMarsden}, \cite{VankerschaverLeok}). A system of $K$ interacting point vortices in two dimensions can be described by the Lagrangian

\begin{equation}
\label{eq: Lagrangian for point vortices}
L(x_1,y_1,\ldots,x_K,y_K,\dot x_1, \dot y_1,\ldots,\dot x_K,\dot y_K) = \frac{1}{2} \sum_{i=1}^K \Gamma_i (x_i \dot y_i - y_i \dot x_i) - H(x_1,y_1,\ldots,x_K,y_K)
\end{equation}

\noindent
with the Hamiltonian

\begin{equation}
\label{eq: Hamiltonian for point vortices}
H(x_1,y_1,\ldots,x_K,y_K) = \frac{1}{4 \pi} \sum_{i < j}^K \Gamma_i \Gamma_j \log \big( (x_i-x_j)^2 + (y_i-y_j)^2 \big)-H_0,
\end{equation}

\noindent
where $(x_i,y_i)$ denotes the location of the $i$-th vortex, $\Gamma_i$ is its circulation, and $H_0$ is an arbitrary constant.

As a test problem we considered the system of $K=2$ vortices with circulations $\Gamma_1=4$ and $\Gamma_2=2$, respectively, and distance $D=1$ between them. The vortices rotate on concentric circles about their center of vorticity at $x_C=0$ and $y_C=0$. We took the initial condition at $x^{(0)}_1= \Gamma_2 D  /(\Gamma_1+\Gamma_2) \approx 0.33$, $y^{(0)}_1=0$, $x^{(0)}_2= -\Gamma_1 D  /(\Gamma_1+\Gamma_2) \approx -0.67$ and $y^{(0)}_2=0$. The analytic solution can be found (see \cite{Newton}) as

\begin{align}
\label{eq: Exact solution for 2 point vortices}
x_1(t) &= \frac{\Gamma_2}{\Gamma_1+\Gamma_2}D\cos \omega t,  \qquad &x_2(t) = -\frac{\Gamma_1}{\Gamma_1+\Gamma_2}D\cos \omega t, \nonumber\\
y_1(t) &= \frac{\Gamma_2}{\Gamma_1+\Gamma_2}D\sin \omega t,  \qquad &y_2(t) = -\frac{\Gamma_1}{\Gamma_1+\Gamma_2}D\sin \omega t,
\end{align}

\noindent
where $\omega = (\Gamma_1+\Gamma_2)/(2 \pi D^2)$. This is a periodic solution with period $T_{period} \approx 6.58$. See Figure~\ref{fig: Reference solution for point vortices}.

\begin{figure}[tbp]
	\centering
		\includegraphics[width=0.8\textwidth]{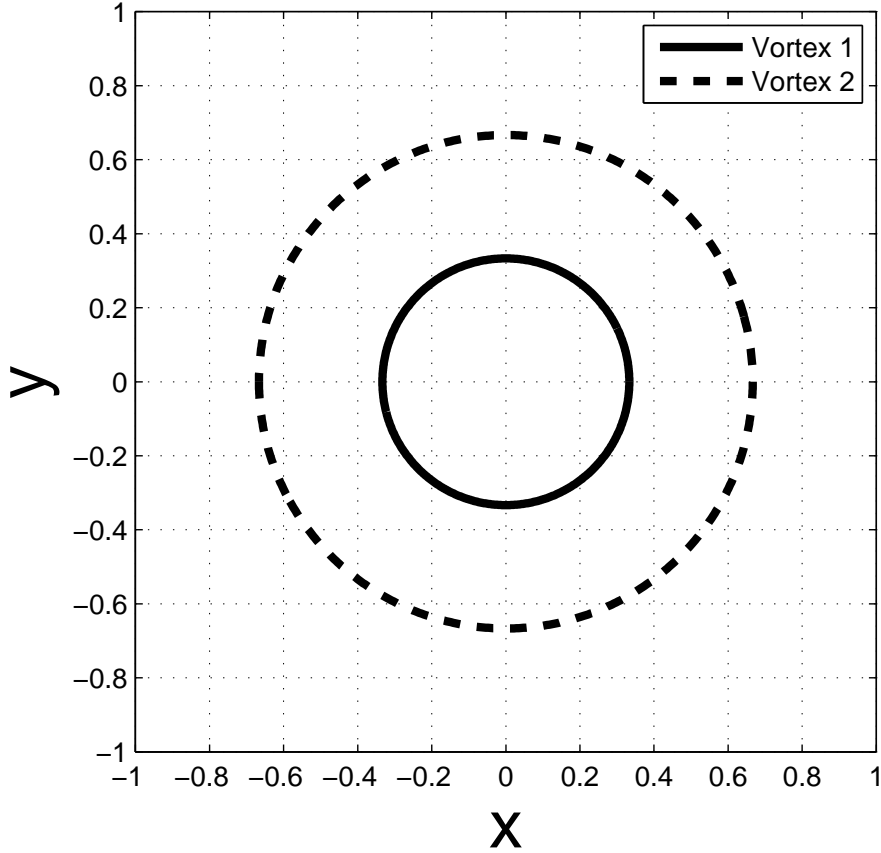}
		\caption{The circular trajectories of the two point vortices rotating about their vorticity center at $x_C=0$ and $y_C=0$.}
		\label{fig: Reference solution for point vortices}
\end{figure}
We performed similar convergence tests as in Section~\ref{sec: Kepler's problem}. The value of the numerical solutions at time T=7 were compared against the exact solution \eqref{eq: Exact solution for 2 point vortices}. The max norm errors are depicted in Figure~\ref{fig: Convergence plot for point vortices}. The results are qualitatively the same as for Kepler's problem.

\begin{figure}[tbp]
	\centering
		\includegraphics[width=0.9\textwidth]{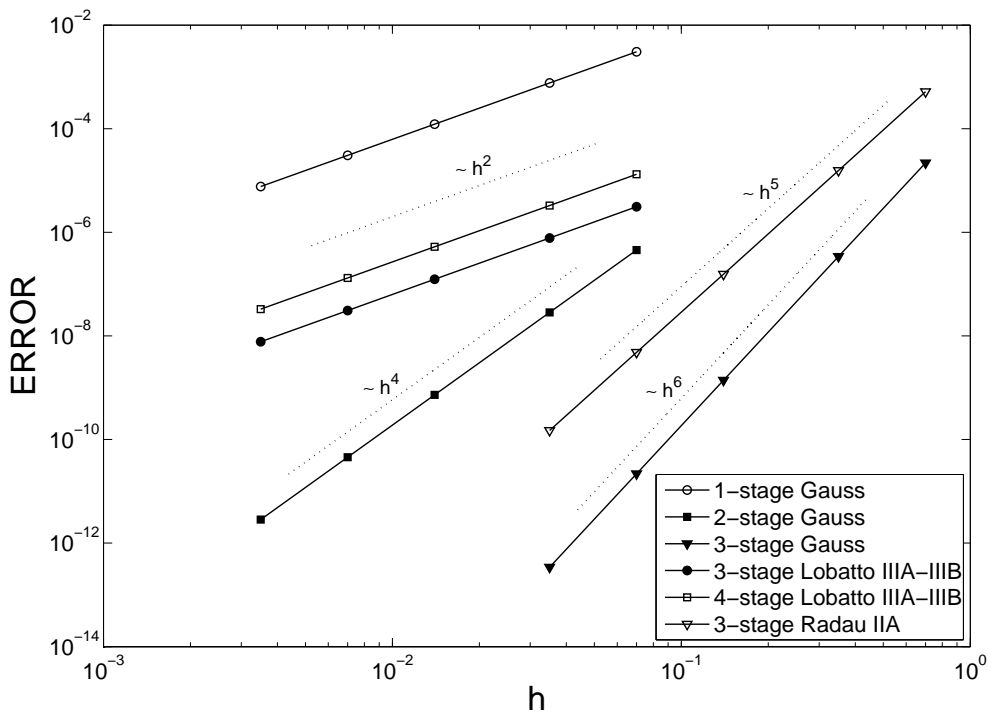}
		\caption{Convergence of several Runge-Kutta methods for the system of two point vortices.}
		\label{fig: Convergence plot for point vortices}
\end{figure}

\begin{figure}[tbp]
	\centering
		\includegraphics[width=0.9\textwidth]{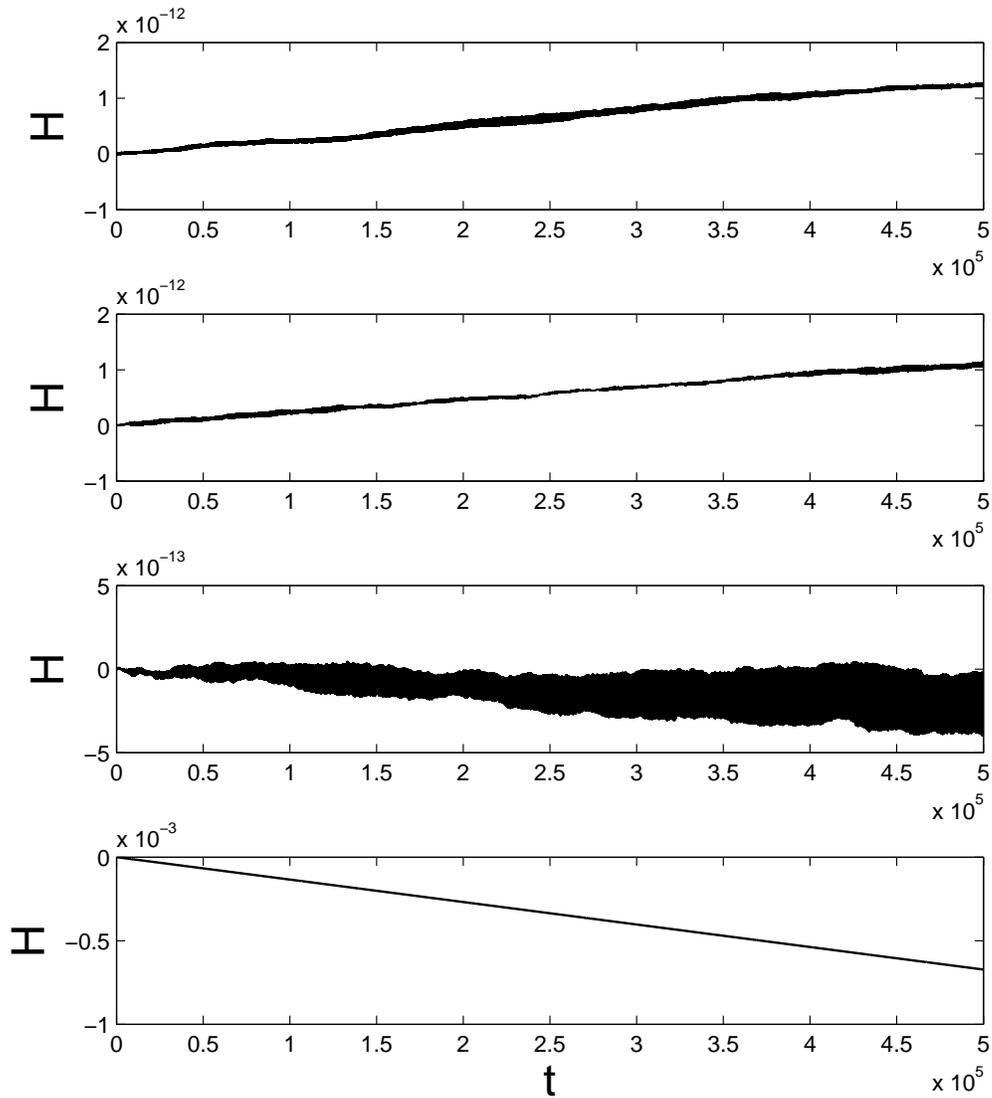}
		\caption{Hamiltonian for the 1-stage (\emph{top}), 2-stage (\emph{second}) and 3-stage (\emph{third}) Gauss, and the 3-stage Radau IIA (\emph{bottom}) methods applied to the system of two point vortices with the time step $h=0.1$ over the time interval $[0,5\times 10^5]$.}
		\label{fig: Energy plots for Gauss methods for point vortices}
\end{figure}

\begin{figure}[tbp]
	\centering
		\includegraphics[width=0.8\textwidth]{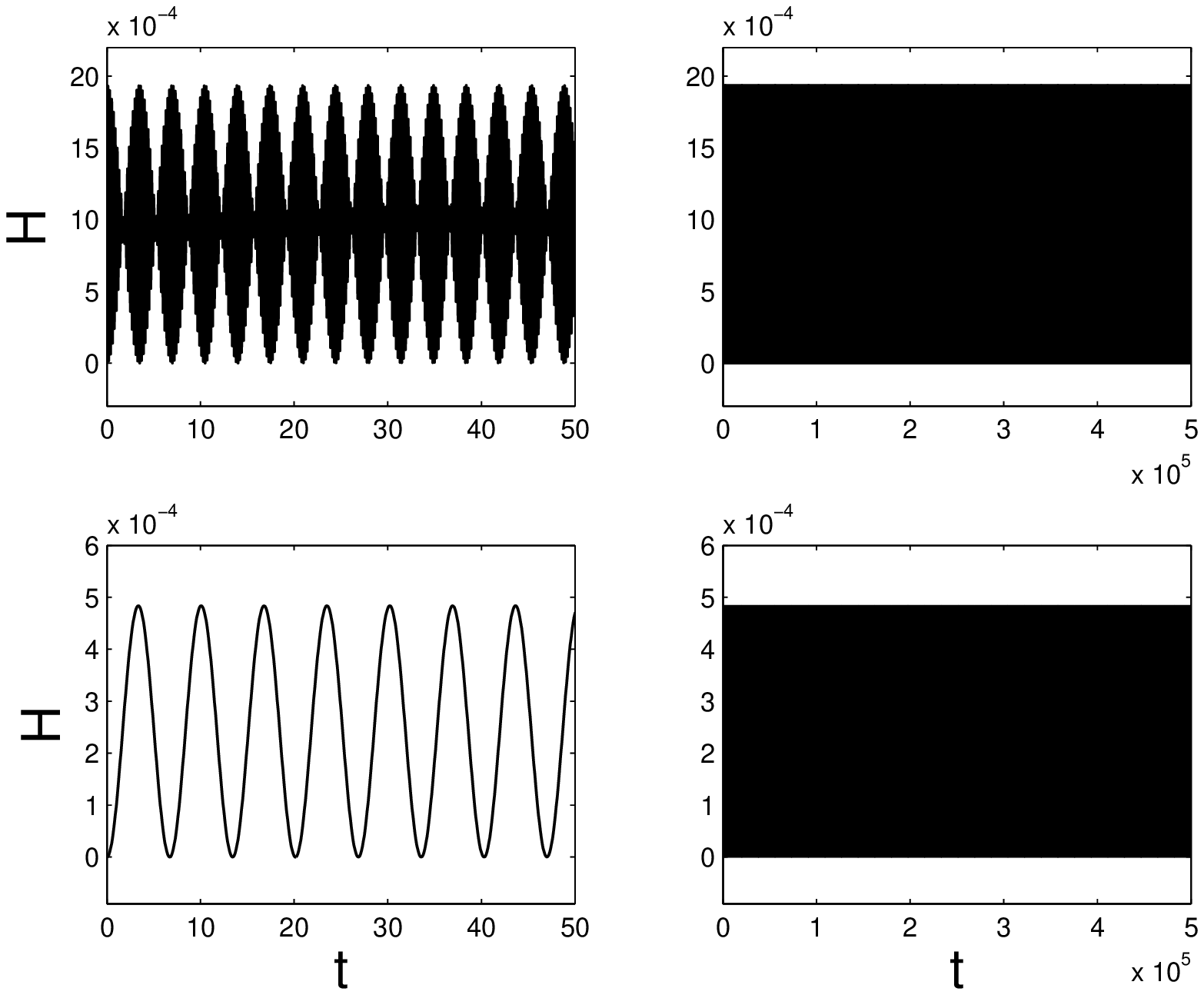}
		\caption{Hamiltonian conservation for the 3-stage (\emph{top}) and 4-stage (\emph{bottom}) Lobatto IIIA-IIIB methods applied to the system of two point vortices with the time step $h=0.1$ over the time interval $[0,5\times 10^5]$ (\emph{right column}), with a close-up on the initial interval $[0,50]$ shown in the \emph{left column}.}
		\label{fig: Energy plots for Lobatto methods for point vortices}
\end{figure}

We set $H_0=0$ in \eqref{eq: Hamiltonian for point vortices}, so that $H=0$ for the considered solution. Figure~\ref{fig: Energy plots for Gauss methods for point vortices} and Figure~\ref{fig: Energy plots for Lobatto methods for point vortices} show the behavior of the numerical Hamiltonian over a long integration interval. The 3- and 4-stage Lobatto IIIA-IIIB integrators performed better than for Kepler's problem. In the case of the Gauss methods the Hamiltonian stayed virtually constant---the visible minor erratic oscillations are the result of round-off errors. The Radau IIA scheme demonstrated a slow but systematic drift.

\subsection{Lotka-Volterra model}
\label{sec: Lotka-Volterra model}

The dynamics of the growth of two interacting species can be modeled by the Lotka-Volterra equations

\begin{align}
\label{eq: Lotka-Volterra equations}
\dot u &= u(v-2), \nonumber \\
\dot v &= v(1-u),
\end{align}

\noindent
where $u(t)$ denotes the number of predators and $v(t)$ the number of prey, and the constants 1 and 2 were chosen arbitrarily. These equations can be rewritten as the Poisson system

\begin{equation}
\label{eq: Lotka-Volterra as a Poisson system}
\left(
\begin{matrix}
\dot u \\
\dot v
\end{matrix}
\right) = \left(
\begin{matrix}
0 & uv \\
-uv & 0
\end{matrix}
\right) DH(u,v),
\end{equation}

\noindent
where the Hamiltonian is given by

\begin{equation}
\label{eq: Hamiltonian for Lotka-Volterra}
H(u,v) = u - \log u + v - 2 \log v - H_0
\end{equation}

\noindent
with an arbitraty constant $H_0$ (see \cite{HLWGeometric}). Using an approach similar to the one presented in Section~\ref{sec: Kepler's problem}, one can easily verify that the Lagrangian

\begin{equation}
L(q,\dot q) = \bigg(\frac{\log q^2}{q^1} + q^2 \bigg) \dot q^1 + q^1 \dot q^2 - H(q)
\end{equation}

\noindent
reproduces the same equations of motion, where $q=(u,v)$. The coordinates $\alpha_\mu(q)$ (cf. Equation~\eqref{eq: Linear Lagrangian in Coordinates}) were chosen, so that the assumptions of Theorem~\ref{thm: Existence of the numerical solution for PRK} are satisfied for the considered Runge-Kutta methods.

As a test problem we considered the solution with the initial condition $q^1_{init}=1$ and $q^2_{init}=1$ (note that $q=(1,2)$ is an equilibrium point). This is a periodic solution with period $T_{period}\approx 4.66$. A reference solution was computed by integrating \eqref{eq:E-L ODE for linear alpha} until the time $T=5$ using Verner's method with the small time step $h = 10^{-7}$. The reference solution is depicted in Figure~\ref{fig: Reference solution for the Lotka-Volterra model}.

\begin{figure}[tbp]
	\centering
		\includegraphics[width=0.8\textwidth]{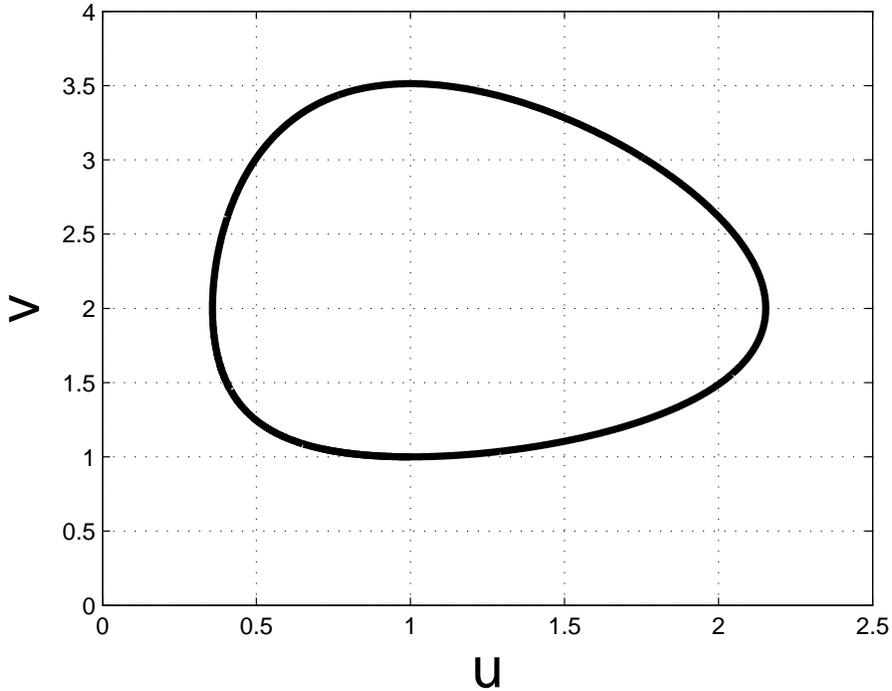}
		\caption{The reference solution for the Lotka-Volterra equations computed by integrating \eqref{eq:E-L ODE for linear alpha} until the time $T=5$ using Verner's method with the time step $h = 10^{-7}$.}
		\label{fig: Reference solution for the Lotka-Volterra model}
\end{figure}

Convergence plots are shown in Figure~\ref{fig: Convergence plot for the Lotka-Volterra model}. The convergence rates for the Gauss and Radau~IIA methods are consistent with the theoretical results presented in Section~\ref{sec: Runge-Kutta methods}---we see that the orders of the 2- and 3-stage Gauss schemes are reduced. The 2-stage Lobatto IIIA-IIIB scheme again proves to be inconsistent, and the 3- and 4-stage schemes converge quadratically, just as in Section~\ref{sec: Kepler's problem} and Section~\ref{sec: Point vortices}.

\begin{figure}[tbp]
	\centering
		\includegraphics[width=0.8\textwidth]{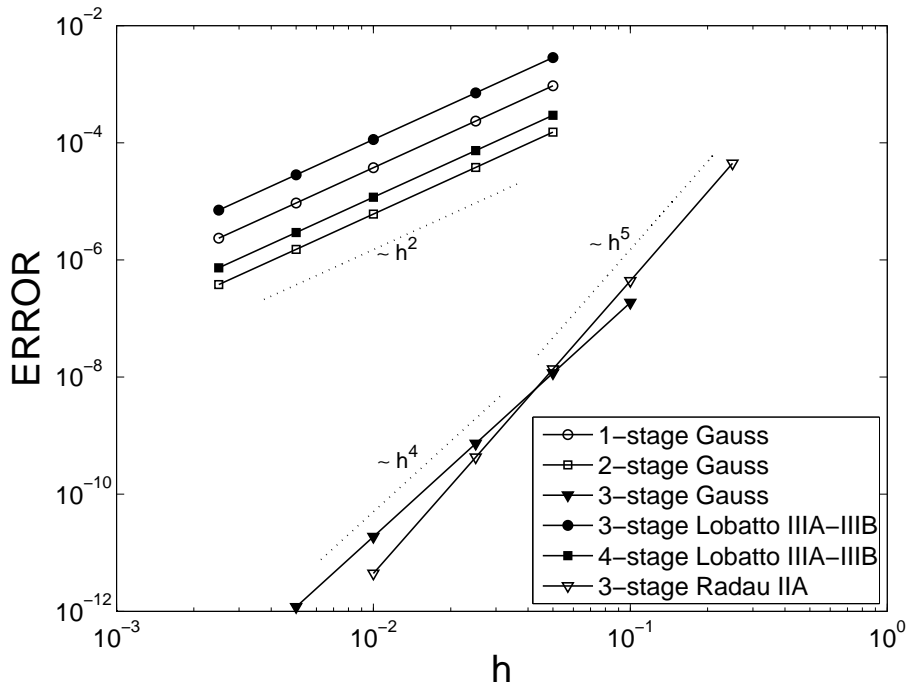}
		\caption{Convergence of several Runge-Kutta methods for the Lotka-Volterra model.}
		\label{fig: Convergence plot for the Lotka-Volterra model}
\end{figure}

We performed another series of numerical experiments with the time step $h=0.1$ to investigate the long time behavior of the considered integrators. The results are shown in Figure~\ref{fig: Energy plots for Gauss methods for the Lotka-Volterra model} and Figure~\ref{fig: Energy plots for Gauss, Lobatto and Radau methods for the Lotka-Volterra model}. We set $H_0=2$ in \eqref{eq: Hamiltonian for Lotka-Volterra}, so that $H=0$ for the considered solution. The 1- and 3-stage Gauss methods again show excellent Hamiltonian conservation over a long time interval. The 2-stage Gauss method, however, does not perform equally well---the Hamiltonian oscillates with an increasing amplitude over time, until the computations finally break down. The Lobatto IIIA-IIIB methods show similar problems as in Section~\ref{sec: Kepler's problem}. The non-variational Radau IIA method yields an accurate solution, but demonstrates a steady drift in the Hamiltonian.

\begin{figure}[tbp]
	\centering
		\includegraphics[width=0.9\textwidth]{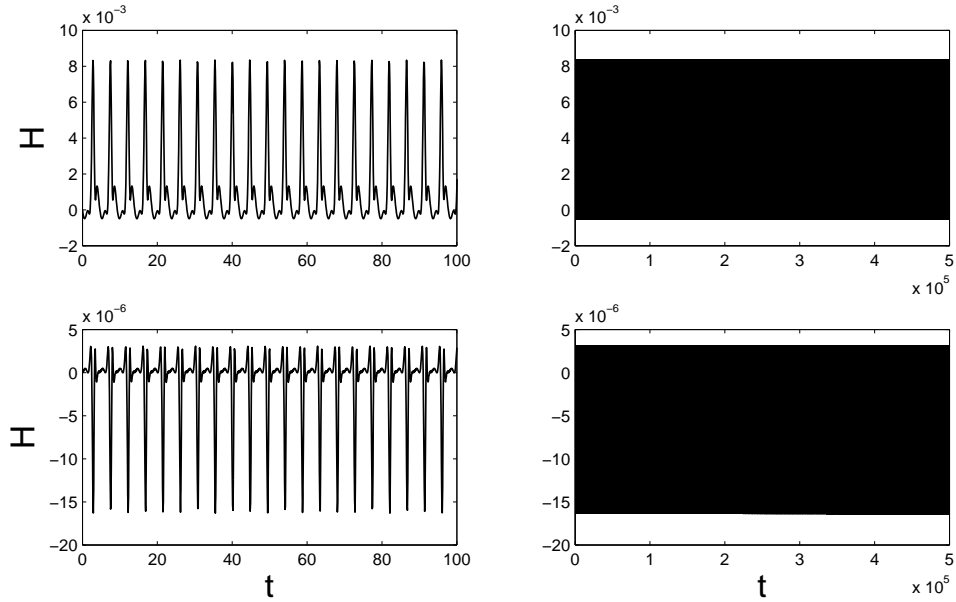}
		\caption{Hamiltonian conservation for the 1-stage (\emph{top row}) and 3-stage (\emph{bottom row}) Gauss methods applied to the Lotka-Volterra model with the time step $h=0.1$ over the time interval $[0,5\times 10^5]$ (\emph{right column}), with a close-up on the initial interval $[0,100]$ shown in the \emph{left column}.}
		\label{fig: Energy plots for Gauss methods for the Lotka-Volterra model}
\end{figure}

\begin{figure}[tbp]
	\centering
		\includegraphics[width=0.9\textwidth]{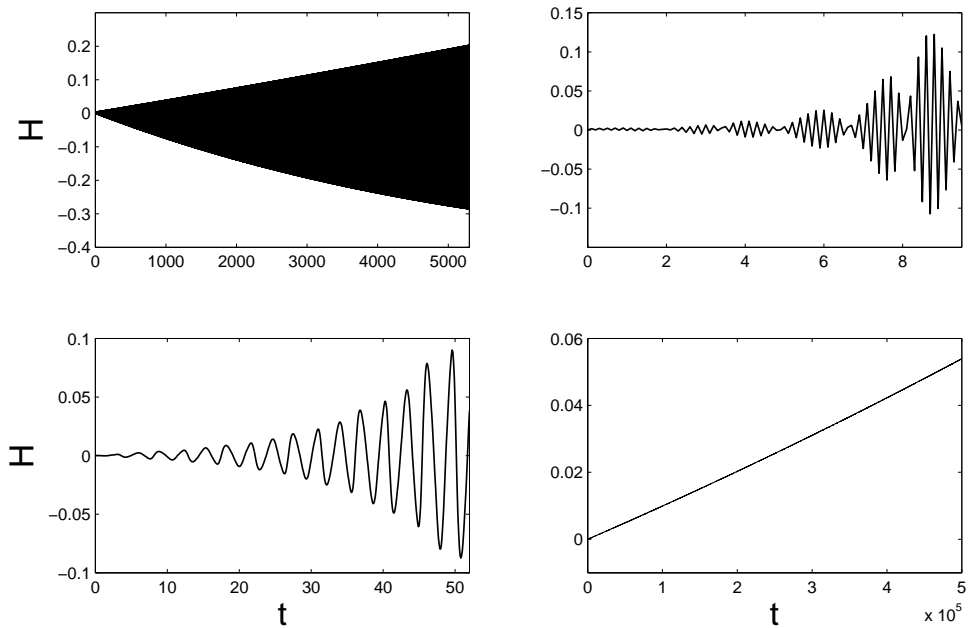}
		\caption{Hamiltonian for the numerical solution of the Lotka-Volterra model obtained with the 2-stage Guass method (\emph{top left}), the 3- and 4-stage Lobatto IIIA-IIIB schemes (\emph{top right} and \emph{bottom left}, respectively), and the non-variational Radau IIA method (\emph{bottom right}).}
		\label{fig: Energy plots for Gauss, Lobatto and Radau methods for the Lotka-Volterra model}
\end{figure}

\section{Summary}
\label{sec: summary}

We analyzed a class of degenerate systems described by Lagrangians that are linear in velocities, and presented a way to construct appropriate higher-order variational integrators. We pointed out how the theory underlying variational integration is different from the non-degenerate case and we made a connection with numerical integration of differential-algebraic equations. We also performed numerical experiments for several example models.

Our work can be extended in several ways. In Section~\ref{sec: Lotka-Volterra model} we presented our numerical results for the Lotka-Volterra model, which is an example of a system for which the coordinate functions $\alpha_\mu(q)$ are nonlinear. The 1- and 3-stage Gauss methods performed exceptionally well and preserved the Hamiltonian over a very long integration time. It would be interesting to perform a backward error (or similar) analysis to check if this behavior is generic. If confirmed, our variational approach could provide a new way to construct geometric integrators for a broader class of Poisson systems.

It would also be interesting to further consider \emph{constrained} systems with Lagrangians that are linear in velocities and construct associated higher-order variational integrators. This would allow to generalize the space-adaptive methods presented in \cite{TyranowskiPHD}, \cite{TyranowskiDesbrunRAMVI} to degenerate field theories, such as the nonlinear Schr\"{o}dinger, KdV or Camassa-Holm equations.

\section*{Acknowledgments}
We would like to thank Prof. Ernst Hairer and Dr. Joris Vankerschaver for useful comments and references. Partial funding was provided by NSF grant CCF-1011944.



\end{document}